\documentclass[11pt]{amsart}

\usepackage{geometry} % page geometry
\usepackage{fullpage} % smaller margins
\usepackage{fontspec} % displays special symbols
\usepackage[main=english, russian]{babel} % language related package
\usepackage{float} % better figures
\usepackage{csquotes} % quotations
\usepackage{wrapfig} % wraps text around small figures
\usepackage{datetime} % dates

\usepackage{subcaption} % subfigures and other subcaptions
\usepackage{multicol} % multiple columns of text
\usepackage[bottom]{footmisc} % footnotes at bottom of page
\usepackage{todonotes} % to do list
\usepackage{hyperref} % clickable links

\usepackage{tikz}
\usetikzlibrary{patterns,arrows}
\tikzset{shorten <>/.style={shorten >=#1, shorten <=#1}}

\usepackage{enumitem} % better enumerate
\setlist[enumerate]{nosep, label=(\roman*), leftmargin=*}

\usepackage[backend=biber, giveninits=true]{biblatex} % bibliography
\usepackage{graphicx} % images
\graphicspath{{images/}}

\usepackage{color} % for adding comments

\usepackage{amsmath, amssymb, amsfonts}
\usepackage{mathtools, stmaryrd}
\usepackage{tikz-cd}
\usepackage{mathUtil}

\usepackage{amsthm}

\renewcommand{\Box}{\boxempty}

\newcommand{\domm}{\mathrm{dom}}

\newcommand{\til}{\mathcal{T}_{\tau}}

\newcommand{\W}{\mathcal{W}}

\newcommand{\ML}{\mathsf{ML}}
\newcommand{\Ho}{\mathsf{H}}
\newcommand{\V}{\mathsf{V}}
\newcommand{\R}{\mathsf{R}}
\newcommand{\So}{\mathsf{S}}
\newcommand{\T}{\mathsf{T}}

\newcommand{\sko}{\mathsf{Skvo}}
\newcommand{\eff}{\mathcal{E}{\mathit{ff}}}
\newcommand{\up}{\mathsf{Up}}
\newcommand{\ipc}{\mathsf{IPC}}

\newcommand{\disji}{%
  \mathrel{\raisebox{0.3pt}{%
    \rotatebox[origin=c]{-90}{\scalebox{0.75}{$\geq$}}}}} % for inquisitive disjunction
\newcommand{\vvvee}{\disji} % for inquisitive disjunction

\makeatletter
\def\l@subsection{\@tocline{2}{0pt}{2.5pc}{5pc}{}}
\def\l@subsubsection{\@tocline{3}{0pt}{5pc}{7.5pc}{}}
\makeatother

\title{Medvedev logic is undecidable}

\author{Rodrigo Nicolau Almeida}
\author{Søren Brinck Knudstorp}


\date{}

\begin{document}

\begin{abstract}
We show that Medvedev's logic of finite problems, a well-known superintuitionistic logic, is undecidable. The key method is a reduction from the periodic tiling problem to non-theoremhood in Medvedev's logic. This settles a longstanding open problem. Using similar techniques, but reducing instead to the ordinary tiling problem, we likewise obtain undecidability of Skvortsov's logic of infinite problems, and the fact that the two logics are distinct -- in fact, they are separated by any aperiodic tiling of the plane. Due to the fact that Medvedev's logic figures in so many different areas, these results have implications for several fields -- for example, the study of schematic fragments of logics such as propositional dependence logic, or the study of internal logics of toposes. 

The core idea and technical work of the undecidability proof were obtained using ChatGPT Sol 5.6, and formally verified in Lean by Claude Opus 5. A detailed methodology section outlines how such results were obtained.
\end{abstract}

\maketitle
\tableofcontents

\section{Introduction}

It is a longstanding open problem in the study of superintuitionistic logics whether Yu Medvedev's logic of finite problems $\ML$ (referred to throughout simply as \emph{Medvedev's logic}), is recursively axiomatizable. In this note, we  settle this problem in the negative. Several adjacent questions, such as the undecidability of Skvortsov's logic $\sko$, are also proved.

We prove undecidability of $\ML$ (respectively, $\sko$) by reduction from tiling problems. These were first introduced by Wang~\cite{Wang1961,Wang1963}, with a milestone undecidability result proven by Berger~\cite{Berger1966-BERTUO-6}. Such problems are well-known tools in proving complexity and computability results in mathematical logic, and have recently been employed by the second-named author to prove a number of results (see \cite{knudstorpundecidabilityrelevant,Knudstorp2024, Knudstorp2025,bunchedimplicationsundecidable}). Nevertheless, due to the semantic tameness of $\ML$ and $\sko$, it appeared to many researchers in the field that such methods could not be carried out for proving undecidability of these logics.

The key tool, allowing the reduction to be carried out, comes in the form of M. Zakharyaschev's \emph{canonical formulas} (originally introduced for $\mathsf{K4}$ \cite{zacharyaschevcanonical}, see also \cite[Ch.~8]{Chagrov1997-cr}), especially in their algebraic/dual reformulation by Bezhanishvili and Bezhanishvili \cite{BEZHANISHVILI2009}. They provide a generalization of Yankov-De Jongh formulas and subframe formulas (see \cite{Bezhanishvili2022jankovformulas} and \cite{ilinphdthesis} for extended discussions), and their interest is twofold: (1) they axiomatize every superintuitionistic logic \cite[Thm.~9.4]{Chagrov1997-cr}; (2) they are defined semantically in terms of transformations of frames, allowing for geometric interpretations of refutation. Such formulas have been employed by Zakharyaschev and his collaborators (see e.g. \cite{zacharyaschevcanonical,ZAKHARYASCHEV1993,Zakharyaschev1994} as well as many of the results reported in \cite[Ch.~8, Ch.~10]{Chagrov1997-cr} amongst other chapters), to provide a number of axiomatization and decidability results in modal and superintuitionistic logics. One crucial feature -- to which we will return below -- is that these formulas are characterized by a \emph{fine control over the disjunctive} logical fragment.

Given the interest of this logic to a wide variety of mathematical logical audiences, we have provided an extended introduction and discussion of the paper in the following way: in Section \ref{sec: brief history} we recap the basic definitions, trace the appearances of Medvedev's logic across several different areas, and recap the state of the art on the subject. In Section \ref{sec: main results and consequences} we outline several key consequences of the main results of the paper, explaining in depth how these results imply the answer to several problems in the literature. The remaining sections are organized as follows. In Section \ref{sec: preliminaries}, we recall the basic definitions and techniques employed in the undecidability results. In Section \ref{sec: tiling posets}, we define the tiling posets and their associated tiling formulas, and discuss their geometric interpretation. In Section \ref{sec: undecidability of med}, we prove the key undecidability result of $\ML$, showing the soundness and completeness of the reduction of tiling to non-validity. In Section \ref{sec: undecidability skortsov} we discuss how a small change in the reduction can equally prove that the logic $\sko$ is undecidable. In Section \ref{sec: discussions and open problems}, we highlight several problems left open or outstanding, and some discussion on future directions in the general study of decidability for these logics. In Section \ref{sec: methodology and context}, we discuss the methodology employed in obtaining the present results, detail our usage of large language models, and present our preliminary understanding of the mathematical academic responsibility of caretaking for these results\footnote{As shown in the publicly available documentation (see Section \ref{sec: methodology and context}), the authors prompted and obtained the undecidability result for $\ML$ on the 4th of September 2026 from ChatGPT 5.6, and communicated this to close colleagues and collaborators. In the subsequent week, they set out to write the current paper, and polish it to a minimum degree, with a focus on the survey, the consequences of this work, and the methodological and ethical implications. In the time since, the preprint \cite{pawlowskimedvedevlogic} announced the undecidability of $\ML$, containing an independent report of the same result. Since that paper has a somewhat narrower focus, using some non-standard terminology and notation, we have opted to still make the present work publicly available, in hopes it might be useful to readers interested in this problem. We apologize for the presently sub-optimal notation and narrative, which we hope to improve in future revisions.}.

\subsection{Brief history of Medvedev logic}\label{sec: brief history}

The original motivation of Yu Medvedev \cite{medvedevfiniteproblems}, which justifies the name the \emph{logic of finite problems}, was an attempt to formalize A. Kolmogorov's \cite{Kolmogoroff1932} interpretation of intuitionistic logic as a calculus of problems. In the series of papers \cite{medvedevfiniteproblems,medvedevinterpretationoflogicalformulasbymeans,medvedevinterpretation}, Medvedev presented the connections between the finite-problem interpretation and the intuitionistic propositional calculus, providing some basic results. Such a formulation, and such investigations, also spanned a further line of research, in the recursion theoretic study of ``mass problems", leading to the known \emph{Medvedev lattice} \cite{Skvortsova1988,sorbimedvedevlattice,Terwijn2008-zp}. Despite its great interest, these approaches pull in a different direction, so we will not discuss them in depth.

In its modern Kripke-based formulation, Medvedev's logic can be defined as:
\begin{equation*}
    \ML\coloneqq \mathrm{Log}(\{(\mathcal{P}([n])\setminus \{\varnothing\}, \supseteq) \mid n\in \omega\}),
\end{equation*}

that is, the set of intuitionistic validities of the finite Kripke frames $M_{n}\coloneqq (\mathcal{P}([n])\setminus \{\varnothing\}, \supseteq)$, called \emph{topless Boolean algebras}, or \emph{Medvedev posets}.

By the above definition, $\ML$ has the finite model property, a property which is usually desirable and leads to amenable metalogical properties; for example, it follows immediately that (the validity problem for) $\ML$ is co-recursively enumerable. However, it was realized from early on that, unlike other intuitionistic sets of validities, Medvedev's logic presented a lot of unique features, and some unique challenges: 
\begin{enumerate}
    \item It was shown by Levin \cite{levinsyntactic} that $\ML$ is a maximal extension of $\ipc$ with the disjunction property.\footnote{We recall (see \cite[Ch.~1/2]{Chagrov1997-cr}) that a superintuitionistic logic $L$ is said to have the \emph{disjunction property} if whenever $\phi\vee\psi\in L$ then either $\phi\in L$ or $\psi\in L$. It is \emph{structurally complete} if every rule admissible over a Hilbert calculus of $L$ is derivable.}
    \item It was shown by Prucnal that $\ML$ is likewise structurally complete \cite{prucnalstructuralcompleteness} (see also \cite{skortsovprucnal}), thereby solving a problem by Harvey Friedman \cite{prucnalpaper}.\footnote{We refer the reader to the recent note of Prenosil \cite{prenosilonlevinandprucnal} for a clear and modern presentation of the results of Levin and Prucnal.} 
    \item In parallel, a landmark work of Maksimova, Skvortsov and Shehtman \cite{maksimovaskortov} showed that $\ML$ is not finitely axiomatizable.
    \item It was shown by Szatkowski~\cite{szatowskimedvedev}, in a vein similar to the original results of Medvedev, that several syntactic fragments of $\ML$ coincide with the corresponding fragment of $\ipc$ -- for instance, their disjunction-free fragments.
\end{enumerate}
\medskip
These early results led to one basic outstanding problem:

\begin{problem}[Decidability of $\ML$]\label{prob:recursive axiomatizability}
    Is $\ML$ recursively axiomatizable? Equivalently, is $\ML$ decidable? 
\end{problem}

The status of Problem \ref{prob:recursive axiomatizability} as a venerable open problem is well-established in the literature, witness \cite[Prob.~16.3]{Chagrov1997-cr} and \cite{Shehtman1990} or for something more recent~\cite[pp.~692]{shapirovskysaveliev}. Note that the existence of undecidable superintuitionistic logics was known as early as the late 70s \cite{shehtmanundecidable}, and hence the possibility of a negative solution was well present in the mind of the authors; nevertheless, a lot of effort in the literature seem to have been devoted to proving $\ML$ decidable.

The difficulty of proceeding further in the analysis of $\ML$, led to several related questions. One line of analysis was initiated by Dmitry Skvortsov, who introduced a closely connected logic, which we deem \emph{Skvortsov's logic}:
\begin{equation*}
    \sko\coloneqq \mathrm{Log}(\{(\mathcal{P}(I)\setminus \{\varnothing\}, \supseteq) \mid I \text{ is a set}\}).
\end{equation*}
Going forward, we use the shorthand $M_{I}\coloneqq (\mathcal{P}(I)\setminus \{\varnothing\}, \supseteq)$. 

As shown in \cite{skvortsov1979logic}, $\sko$ is recursively axiomatizable, and it follows easily from the definitions that $\sko\subseteq \ML$. Two related problems arise:

\begin{problem}[Coincidence of $\ML$ and $\sko$]\label{prob: coincidence problem}
    Does the inclusion $\ML\subseteq\sko$ hold?
\end{problem}

\begin{problem}[Decidability of $\sko$]\label{prob: skortsov logic decidability}
    Does the logic $\sko$ have the finite model property? Is it decidable?
\end{problem}
Note that a positive answer to the former would imply positive answers to the latter, as well as to Problem~\ref{prob:recursive axiomatizability}.

Further variations on the same idea were studied in \cite{Shehtman1986}, where several approximations were given to separating the two logics; in particular it was shown that formulas in one variable cannot separate the two logics, and that a large class of ``logics of information type" are likewise non-finitely axiomatizable.\footnote{Note, however, that as consequence relations $\sko$ and $\ML$ are known to differ, as the former is compact \cite{skvortsov1979logic} whilst the latter is not \cite[Prop.~3]{Chen2025}.} Also in the same paper, several topological models were noted as being naturally related to these logics.

In a similar vein, in \cite{Shehtman1990} a modal counterpart, $\ML_{\Box}$, of Medvedev's intermediate logic was introduced, defined on the same class of frames but in the basic modal language (and likewise there's $\sko_{\Box}$). $\ML_{\Box}$ extends Grzegorczyk's logic, $\mathsf{Grz}$, and is connected to $\ML$ by the celebrated Gödel--McKinsey--Tarski translation \cite{Godel1932-GDEZIA,McKinsey1944}. The translation, $\sigma\colon \mathcal{L}_{\mathsf{IPC}}\to \mathcal{L}_{\mathsf{S4}}$, goes from intuitionistic to classical modal logic and picks for each superintuitionistic logic a family of \emph{modal companions} which all share the same translated intuitionistic fragment.\footnote{For a detailed treatment of modal companions of superintuitionistic logics, see the survey \cite{Chagrov1992}. For recent developments in the theoretical understanding of these topics, see also \cite{BEZHANISHVILI2025}.} In particular, the greatest modal companion is denoted by
\begin{equation*}
    \sigma(L)\coloneqq \mathsf{Grz}\oplus \{\sigma(\phi) \mid \phi\in L\}.
\end{equation*}
The correctness of the translation, amounts then, to the following theorem
\begin{equation*}
    \phi\in L \iff \sigma (\phi)\in \sigma(L),
\end{equation*}
obtained in this form by Esakia \cite{esakiapaperatconference}. Using this theory, Shehtman's key results \cite{Shehtman1986} showed that $\ML_{\Box}$ was also not finitely axiomatizable, and provided a tighter bridge between the two logics. Moreover, this work proposed a concrete avenue for establishing (un)decidability of the logic $\ML$: what Shehtman deemed $\chi$-completeness, that is, the question of axiomatizability of $\ML$ by Yankov formulas\footnote{Following the convention from \cite{yankovoutstanding}, we transliterate the cyrillic to `Yankov', though `Jankov' is also common.}. Thus, the following were left there as open problems:

\begin{problem}[Decidability of $\ML_{\Box}$]\label{prob: decidability of modal medvedev}
    Is the logic $\ML_{\Box}$ recursively axiomatizable? Equivalently, is $\ML_{\Box}$ decidable?
\end{problem}

\begin{problem}[Decidability of $\sko_\Box$]\label{prob: modal skortsov logic decidability}
    Is $\sko_\Box$ decidable?
\end{problem}

\begin{problem}[Yankov axiomatizability of $\ML$ and $\ML_{\Box}$]\label{prob: yankov axiomatizability}
    Are the logics $\ML$ or $\ML_{\Box}$ axiomatizable by Yankov formulas?
\end{problem}
(Note that, contrary to the intuitionistic language, the modal language can express Grzegorczyk's formula, and hence $\sko_\Box$ lacks the FMP and the inclusion $\sko_\Box\subsetneq \ML_\Box$ is strict.)

Several more recent results are on (classes of) logics close to $\ML$. In \cite{shapirovskysaveliev} an approach in terms of products of ordinals was given; in \cite{Grilletti2022} it was shown that $\ML$ coincides with the logic of topless finite distributive lattics, i.e. $\mathrm{Log}(\{D\setminus\top \mid D\text{ is a finite distributive lattice} \})$; in~\cite{Knudstorp2025} the second-named author proves undecidability for a conservative extension of~$\sko_\Box$;
very recently, the non-finite-axiomatizability results of \cite{maksimovaskortov} were extended to a wide class of logics \cite{xiaogeneralizedmedvedev}; and in a different direction, logics of the individual Medvedev frames $M_{n}$, and related ``Strong Union Logics", were also studied \cite{Chen2025,Chen2026}, obtaining structural completeness results.

In advancing a variety of approaches to the study of $\ML$, these results continued to orbit around the central questions of decidability and axiomatizability laid above. At the same time though, in parallel with these analyses, the logic $\ML$ was found to make appearances in several other, seemingly unexpected, areas of mathematical logic. We record these here, to the best of our knowledge:
\begin{enumerate}
    \item In a study of topological semantics of modal logic, Bezhanishvili, Gehrke and Van Benthem \cite{vanBenthem2003} introduced the ``logic of chequered subsets of the reals", $\mathsf{Cheq}$, which was observed in \cite{litakcheq} to be a sublogic of $\ML$. Subsequent research revealed that $\mathsf{Cheq}$ had, like $\mathsf{ML}$, quite an intricate structure \cite{fontainecheq,inproceedings}. This line of research was also extended by Xiao \cite{xiaomedvedev}, who connected $\ML$ with the modal logic of forcing \cite{Hamkins2007}.
    \item In \cite{Ciardelli2011-CIAIL}, a key connection was established between \emph{inquisitive logic} and $\ML$, showing that Medvedev's logic is the \emph{schematic fragment} of inquisitive logic \cite[Thm.~4.4]{Ciardelli2011-CIAIL}. We recall that given a superinquisitive logic (or more generally, a set of formulas closed under Modus Ponens but not necessarily uniform substitution), this is defined as
    \begin{equation*}
        \mathsf{Sch}(L)=\{\phi\in \mathcal{L}_{\mathsf{IPC}} \mid \text{for all substitutions } \sigma\text{, } \ \sigma(\phi)\in L\}.
    \end{equation*}
    
    Related results, without the framing of inquisitive logic, were independently obtained by Ferrari and Miglioli \cite{ferrarimiglioli}. This connection was reaffirmed in the work of Holliday \cite{hollidaymedvedev}, and later algebraic work on inquisitive logic \cite{Bezhanishvili2019,BEZHANISHVILI2021}. 
    \item In \cite{bevilacqua2025medvedev} (see also \cite{jerabekinternal}) it was observed that $\ML$ appears also in the context of \emph{categorical logic}. To clarify this connection, let us recall the Mitchell-Bénabou language of elementary toposes (see \cite[Ch.~VI.5]{MacLane1994}); importantly for us, one defines an interpretation of intuitionistic formulas $\phi(p_{1},\dots,p_{n})$ in an elementary topos $\mathcal{E}$, using the subobject classifier $\Omega$. Formulas are interpreted as maps $\tru{\phi}\colon \Omega^{n}\to \Omega$, where variables are projections $\pi_{i}\colon \Omega^{n}\to \Omega$, and  the operations are those of $\Omega$ as an internal Heyting algebra. Such a map is true when it evaluates to the `truth map' $t\colon 1\to \Omega$, which is part of the structure of a subobject classifier, and in this case, we write $\mathcal{E}\vDash \phi$. The \emph{internal logic} is then given as
    \begin{equation*}
        \mathrm{Log}(\mathcal{E}) = \{\phi\in \mathcal{L}_{\mathsf{IPC}} \mid \mathcal{E}\vDash \phi\}.
    \end{equation*}
    The key result observed in the above works, is that $\mathrm{Log}(\mathbf{sSet})=\ML$, where $\mathbf{sSet}$ is the topos of simplicial sets.
\end{enumerate}

\vspace{3mm}

Out of this large body of work, several related problems have arisen. We will return to some of these problems, which bear close connection to the work of this paper, in Section \ref{sec: discussions and open problems}.

\subsection{Main results and consequences}\label{sec: main results and consequences}

The main contribution of the present paper is to provide a negative answer to several of the problems discussed in the previous section. The first and central result is the following negative answer to Problem \ref{prob:recursive axiomatizability}, given in Theorem \ref{thm: medvedev's logic is undecidable}:

\begin{theorem}
    The logic $\ML$ is undecidable. In particular, it encodes a periodic tiling problem.
\end{theorem}

From this we obtain several corollaries, starting from the following negative answer to Problem \ref{prob: decidability of modal medvedev}:

\begin{corollary}
    The logic $\ML_{\Box}$ is undecidable.
\end{corollary}
\begin{proof}
    If $\ML_{\Box}$ was decidable, a decidability algorithm for $\ML$ would follow, by computing the translation $\sigma(\phi)$ of any formula.
\end{proof}

Due to its role in inquisitive logic, the negative answer to Problem \ref{prob:recursive axiomatizability} immediately entails:

\begin{corollary}
    The schematic fragment of inquisitive logic $\mathsf{Inq}$ is undecidable.
\end{corollary}
It additionally follows that the problem of deciding, given an inquisitive validity $\varphi$, whether all of its substitutional instances $\sigma(\varphi)$ are validities too, is undecidable. We will return to this in greater detail in Subsection~\ref{subsec:schematic}.

Likewise, due to the connections between $\ML$ and the topos of simplicial sets, we obtain:

\begin{corollary}
    The internal propositional logic of the presheaf topos $\mathbf{sSet}$ is undecidable.
\end{corollary}

We also obtain an answer to problems concerning the logic $\sko$. The first answers the relationship between this logic and $\ML$, answering in the negative Problem \ref{prob: coincidence problem}, and arose from the periodicity of the tiling:

% \rodrigonote{Check if we in fact have this result}

\begin{theorem}
    The logic $\sko$ is strictly contained in $\ML$. In particular, there is a formula $\rho_{\mathcal{W}}$, associated to a periodic set of tiles, such that $\rho_{\mathcal{W}}\in \ML\setminus \sko$.
\end{theorem}

The key to this theorem lies in realizing that the finite Medvedev frames force periodicity in tiles, whilst the infinite Medvedev frames, namely $M_{\omega}$, need not induce any cycles. This aperiodicity suggested to us the following answer to Problem \ref{prob: skortsov logic decidability} and~\ref{prob: modal skortsov logic decidability}:

\begin{theorem}
    The logics $\sko$ and $\sko_\Box$ are both undecidable. In particular they encode a (non-periodic) tiling problem.
\end{theorem}

\section{Preliminaries}\label{sec: preliminaries}

In this section we recall all basic facts and notation used in the undecidability proof. The reader is referred to \cite{Chagrov1997-cr} for most general facts about superintuitionistic and modal logics discussed here.

Throughout we use $\mathcal{L}_{\mathsf{IPC}}$ to denote the language of intuitionistic logic, and $\mathcal{L}_{\mathsf{ML}}$ to denote the language of normal modal logic. We say that a set of formulas $L$ is a \emph{logic} if it is closed under Modus Ponens and uniform substitution. A \emph{superintuitionistic logic} is a logic in the language $\mathcal{L}_{\mathsf{IPC}}$ such that $\mathsf{IPC}\subseteq L$. We denote by $\mathrm{Ext}(L)$ the lattice of extensions of $L$. Similarly, a \emph{normal modal logic} is a logic in the language $\mathcal{L}_{\mathsf{ML}}$ such that $\mathsf{K}\subseteq L$ which in addition is closed under necessitation.

We will work throughout with posets $(P,\leq)$, denoted simply as $P$. Such a poset is \emph{rooted} if there is an element $r\in P$ (the root) such that for each $p\in P$, $r\leq p$. Given $p\leq q$ we frequently refer to $q$ as a \emph{successor} of $p$; it is an \emph{immediate successor} if $p<q$ and there is no $q'$ such that $p< q'<q$. A \emph{chain} will mean a sequence of pairwise comparable points, whilst an \emph{antichain} will mean a sequence of pairwise incomparable points.

A subset $U\subseteq P$ is called \emph{upwards-closed} (often simply an \emph{upset}) if whenever $x\in U$ and $x\leq y$ then $y\in U$. We denote the set of subsets of $P$ by $\up(P)$. Given a subset $C\subseteq P$ we write
\begin{equation*}
    \Box C\coloneqq \{p\in P : \forall q(p\leq q \Rightarrow q\in C)\}.
\end{equation*}
Given two upsets $U,V$ we write simply $U\rightarrow V\coloneqq \Box (P\setminus U\cup V)$. Note that $U\rightarrow V$ is an upset.

Given a subset $C\subseteq P$ we will write $\min C$ to mean the set of minimal elements with respect to the order in $C$; similarly we will write $\max C$. Note that for any subset $C$, $\min C$ and $\max C$ are always antichains.

\subsection{Heyting algebras and discrete duality}

To understand the structure of canonical formulas, it will be valuable to understand the algebraic semantics of superintuitionistic logics:

\begin{definition}[Heyting algebra]
    A bounded lattice $(H,\leq)$ is called a \emph{Heyting algebra} if for each $a,b\in H$, there is $a\rightarrow b\in H$ such that for each $c\in H$
    \begin{equation*}
        a\wedge c\leq b\iff c\leq a\rightarrow b.
    \end{equation*}
\end{definition}

Heyting algebras, in the signature $(\wedge,\vee,\rightarrow,0,1)$, form a variety of algebras, and a category, denoted $\mathbf{HA}$. The subcategory of its finite algebras is denoted $\mathbf{FinHA}$.  Through this semantics, validity of formulas is therefore turned into an equational constraint: given a formula $\phi(p_{1},\dots,p_{n})$ we write $H\vDash \phi$ to mean that for each $a_{1},\dots,a_{n}$ we have
\begin{equation*}
    H\vDash \phi(a_{1},\dots,a_{n})=1.
\end{equation*}

Given a class $\mathbf{A}$ of algebras, we often talk of the \emph{logic of $A$},
\begin{equation*}
    \mathrm{Log}(\mathbf{A})=\{\phi \in \mathcal{L}_{\mathsf{IPC}} : \forall H\in \mathbf{A}, \ H\vDash \phi\}.
\end{equation*}

In this paper we will work mostly with special Heyting algebras arising from posets. For these, a duality theory exists developed by De Jongh and Troelstra \cite{DeJongh1966}, which stands in parallel to the duality developed by Esakia\footnote{For a book-length treatment we refer the reader to \cite{Esakiach2019HeyAlg}.} \cite{esakiatopologicalkripkemodels}.

\begin{definition}[Upset algebra]
    Given a poset $P$, the \emph{upset algebra} of $P$ is the Heyting algebra $(\up(P),\cap,\cup,\varnothing,P,\rightarrow)$.
\end{definition}

We denote by $\mathbf{HA}_{\infty}$ the (non-full) subcategory of Heyting algebras which are upset algebras, together with complete Heyting homomorphisms. 

Note that in general, given an order-preserving map $f\colon P\to Q$, $f^{-1}\colon \up(Q)\to \up(P)$ maps upsets to upsets, and is a distributive lattice homomorphism. But it may fail to be a Heyting homomorphism:

\begin{definition}[P-morphism]
    An order-preserving map $f\colon P\to Q$ is called a \emph{p-morphism} if for all $p\in P$, and $q\in Q$ whenever $f(p)\leq q$ then there is some $p'\geq p$ such that $f(p')=q$.
\end{definition}

\begin{proposition}[P-morphisms correspond to Heyting homomorphisms]
    If $f\colon P\to Q$ is a p-morphism, then $f^{-1}\colon \up(Q)\to \up(P)$ is a complete Heyting algebra homomorphism.
\end{proposition}

Put together, this assignment defines a functor $\up\colon \mathbf{Pos}_{p}\to \mathbf{HA}_{\infty}^{\mathrm{op}}$ from the category of finite posets with p-morphisms to the opposite of the category of finite Heyting algebras and their homomorphisms. There likewise exists a functor back, by  considering the \emph{completely prime spectrum} of the algebra $\mathsf{Spec}_{\infty}(H)$, which picks out the collection of completely prime filters (equivalently, principal filters generated by completely join-prime elements) and orders them by inclusion. One likewise has that given a complete Heyting algebra homomorphism $f\colon H\to H'$, the assignment $f^{-1}\colon \mathsf{Spec}_{\infty}(H')\to \mathsf{Spec}(H)$ is always a p-morphism, defining a functor $\mathsf{Spec}_{\infty}\colon \mathbf{HA}\to \mathbf{Pos}_{p}^{\mathrm{op}}$. Indeed, we have \cite[pp.329]{DeJongh1966}:

\begin{proposition}[De Jongh-Troelstra duality]\label{prop: finite Esakia duality}
    The functors $\up$ and $\mathsf{Spec}_{\infty}$ define a dual equivalence.
\end{proposition}

This equivalence will be used freely in the subsequent section.

\subsection{Canonical formulas}

An important ingredient below is the usage of Zakharyaschev's canonical formulas for superintuitionistic logic \cite{zacharyaschevcanonical}. Our particular formulations, of a more algebraic nature, are taken from \cite{BEZHANISHVILI2009}, where these formulas were reinterpreted in light of Esakia duality.

\begin{definition}[Opremum]
    Given a Heyting algebra $H$ we say that an element $o\in H$ is the \emph{opremum} if $o$ is the second largest element of $H$.
\end{definition}

\begin{remark}[Subdirectly irreducible = Opremum = Rooted poset]
    It is well-known \cite[A.1.2]{Esakiach2019HeyAlg} that a Heyting algebra is subdirectly irreducible if and only if it has an opremum. Moreover, under Esakia duality, a finite Heyting algebra $H$ is subdirectly irreducible if and only if $\mathsf{Spec}(H)$ is rooted. We will make free use of this in what follows.
\end{remark}

\begin{definition}[Canonical formula]
    Let $A$ be a finite subdirectly irreducible Heyting algebra, with $o$ its opremum, and let $D\subseteq A^{2}$, called the \emph{closed domain}. For each $a\in A$ we  introduce a variable $p_{a}$, and define,
\begin{align*}
\Delta(A,D) \coloneqq 
 & (p_{0}\leftrightarrow 0)\wedge (p_{1}\leftrightarrow 1) \wedge \bigwedge   \{p_{a\ast b}\leftrightarrow(p_a\ast p_b) : a,b\in A, \ast\in \{\wedge,\rightarrow\}\} \\
 &\wedge \bigwedge \{p_{a\vee b}\leftrightarrow(p_a\vee p_b) : (a,b)\in D\}.
\end{align*}
The \emph{canonical formula} associated with the pair $(A,D)$ is the formula:
\begin{equation*}
    \alpha(A,D)\coloneqq \Delta(A,D)\rightarrow p_{o}.
\end{equation*}
\end{definition}

Intuitively, a canonical formula consists of a diagram formula. The usage of the closed domains arises out of the fact that Heyting algebras are not locally finite: one may generate infinitely many formulas already with one generator.

\begin{notation}
    Using the duality of Proposition \ref{prop: finite Esakia duality} -- and by abuse of notation -- given a finite rooted poset $P$ and $D\subseteq \up(P)^{2}$, we will write $\alpha(P,D)$ for the canonical formula $\alpha(\up(P),D)$.
\end{notation}

\begin{remark}[Yankov formulas and stable canonical formulas]
    Several related classes of formulas are connected to canonical formulas.
    \begin{enumerate}
        \item For a finite subdirectly irreducible Heyting algebra $A$, a canonical formula $\alpha(A,D)$ is called a \emph{Yankov formula} whenever $D=A^{2}$. In such a case we denote the formula by $\mathcal{J}(H)$, or $\mathcal{J}(P)$ when referring to the associated finite poset.
        \item Canonical formulas exploit the local tabularity of the $(\wedge,\rightarrow,\top,\bot)$-fragment of $\mathsf{IPC}$, due to Diego \cite{Diego1966-DIESLA-5}. If instead one exploits the local tabularity of the $(\wedge,\vee,\top,\bot)$-fragment, the equivalent of canonical formulas are the \emph{stable canonical formulas}\cite{stablecanonicalformulas}. The choice to work with canonical formulas made here lies in the fact that these seem to make the definition of some maps in the reductions given in Section \ref{sec: undecidability of med} substantially easier.
    \end{enumerate}
\end{remark}

The classical Yankov-De Jongh theorem \cite[Lem.~3.3]{Bezhanishvili2022jankovformulas} establishes that for any Heyting algebra $H$, $A$ a finite subdirectly irreducible Heyting algebra, we have
\begin{equation*}
    H\nvDash \mathcal{J}(A) \iff \exists H_{0},\text{ such that} \  H\twoheadrightarrow H_{0}, \text{ and } A\hookrightarrow H_{0},
\end{equation*}
that is, a homomorphic image of $H$ such that $A$ embeds in it. In symbols, $A\in \mathbb{HS}(H)$. Put in different terms, there is a subdirectly irreducible factor of $H$ into which $A$ embeds as a Heyting algebra.

The idea of passing to canonical formulas is to demand that only the joins in $D$ be calculated correctly. This translates to the following algebraic meaning \cite[Thm.~5.3]{BEZHANISHVILI2009}:

\begin{proposition}[Algebraic criterion for validating canonical formulas]\label{prop: algebraic criterion for validating canonical formulas}
    For each Heyting algebra $H$, and each finite subdirectly irreducible Heyting algebra $A$, with opremum $o$, and $D\subseteq A^{2}$ we have that $H\nvDash \alpha(A,D)$ if and only if there is a homomorphic image $H_{0}$ of $H$ a \emph{bounded implicative semilattice embedding}\footnote{Recall that such homomorphisms are those which preserve the $(\wedge,\rightarrow,0,1)$ subreduct.} $h\colon A\to H_{0}$, such that for each $(a,b)\in D$ we have:
    \begin{equation*}
        h(a\vee b)=h(a)\vee h(b).
    \end{equation*}
\end{proposition}

Using a version of Proposition \ref{prop: algebraic criterion for validating canonical formulas}, Zakharyaschev proved the following uniform axiomatization result \cite{zacharyaschevcanonical} (see also \cite[Cor.~5.10]{BEZHANISHVILI2009}):

\begin{theorem}[Canonical formula axiomatization]
    For every superintuitionistic logic $L$, $L$ is axiomatized by canonical formulas.
\end{theorem}

For our purposes we will need an extension of Esakia duality to these classes of maps.

\begin{notation}
Given $P$ a poset, $Q$ a finite poset and let $f\colon P\rightharpoonup Q$ \emph{partially defined} map we will write $\domm(f)$ for its domain, and will write $f[{\uparrow}x]$ for the image under $f$ of all points in $\domm(f)\cap {\uparrow}x$. 
\end{notation}

\begin{definition}[Partial Esakia morphisms]\label{def:partial-esakia}
A partial map $f\colon P\to Q$ between finite posets is called a \emph{partial Esakia morphism} if the following conditions hold:\vspace{.2cm}
\begin{itemize}
    \item \makebox[3.5cm][l]{(Order-preserving)} If $p,p'\in \domm(f)$ and $p\leq p'$ then $f(p)\leq f(p')$.\vspace{.15cm}
    \item \makebox[3.5cm][l]{(Back-condition)} If $p\in \domm(f)$ and $f(p)\leq q$ then there is some $p'\geq p$ such that $p'\in \domm(f)$ and $f(p')=q$.\vspace{.15cm}
    \item \makebox[3.5cm][l]{(Domain definition)} $p\in \domm(f)$ if and only if $f[{\uparrow}p]={\uparrow}q$ for some $q\in Q$.\vspace{.2cm}
\end{itemize}
Such a map is called \emph{cofinal} if whenever $p\in P$ then there is some $p'\geq p$ such that $p'\in \domm(f)$.
\end{definition}

\begin{definition}[Closed domain condition]\label{def: closed domain condition}
    Let $Q$ be a finite rooted poset, $P$ a poset, and let $\mathfrak{D}$ be a collection (possibly empty) of antichains of $Q$. A cofinal partial Esakia morphism $f\colon P\to Q$ is said to satisfy the \emph{closed domain condition} (CDC) with respect to $\mathfrak{D}$ if:
    \begin{equation*}
        x\notin \domm(f) \implies \min f[{\uparrow}x]\notin \mathfrak{D}.
    \end{equation*}
\end{definition}

Given $P$ a finite rooted poset, $U,V\in \up(Q)$, we write:

\[
 \mathfrak{D}_{U,V}=
 \left\{C\subseteq U\cup V:
 \begin{array}{l}
 C \text{ is an antichain},\\
 C\cap(U\setminus V)\neq\varnothing,\\
 C\cap(V\setminus U)\neq\varnothing
 \end{array}\right\};
\]

then, given $D\subseteq \up(Q)^{2}$, we write
\begin{equation*}
    \mathfrak{D}_{D}\coloneqq \bigcup_{(U,V)\in D} \mathfrak{D}_{U,V}.
\end{equation*}

From all of this we extract the following criterion for a Heyting algebra\footnote{A careful observation of the proofs of \cite{BEZHANISHVILI2009} shows that the assumptions do not rely on finiteness of the poset $Q$; the version provided here therefore removes this assumption, since we will need the slight additional generality for Section \ref{sec: undecidability skortsov}.} to validate a canonical formula \cite[Lem.~3.40, Cor.~5.5]{BEZHANISHVILI2009}: 

\begin{theorem}[Validity of canonical formulas via posets]\label{thm: validity of canonical formulas via posets}
    For each finite poset $Q$, and canonical formula $\alpha(P,D)$ we have that $\up(Q)\nvDash \alpha(P,D)$ if and only if there is an upset $S\subseteq Q$, and an onto cofinal partial Esakia morphism $f\colon S\to P$ satisfying the CDC with respect to $\mathfrak{D}_{D}$.
\end{theorem}

\subsection{Medvedev posets}

\begin{definition}[Medvedev poset]
    Let $I$ be an arbitrary set. We denote by
    \begin{equation*}
        M_{I}\coloneqq (\mathcal{P}(I)\setminus \{\varnothing\}, \supseteq);
    \end{equation*}
    we refer to this as a \emph{topless Boolean algebra} or a \emph{Medvedev poset}. We write simply $M_{n}$ for $(\mathcal{P}([n])\setminus \{\varnothing\}, \supseteq)$.
\end{definition}

Medvedev posets enjoy a lot of unique properties, which will be used throughout this paper. We list them in the following easy Proposition:

\begin{proposition}[Facts about Medvedev posets]\label{prop: facts about medvedev posets}
    The following hold of Medvedev posets.
    \begin{enumerate}
        \item For each $x,y\in M_{n}$, the union $x\cup y\in M_{n}$, that is, $M_{n}$ is a (meet-)semilattice.
        \item \label{eq: hereditariness of Medvedev posets} For each $x\in M_{n}$, there is some $m\leq n$ such that ${\uparrow}x\cong M_{m}$.
    \end{enumerate}
\end{proposition}

Using this Proposition we obtain the following sharpening of Theorem \ref{thm: validity of canonical formulas via posets}.

\begin{definition}[Medvedev cover]
    Given $P$ a finite rooted poset, and $D\subseteq \up(P)^{2}$ we denote by $\mathrm{Cov}_{\ML}(P,D)$ the statement: there is some $n\in \omega$ and a cofinal partial Esakia morphism onto $M_{n}\rightarrow P$ satisfying the CDC for $\mathfrak{D}_{D}$.
\end{definition}

\begin{theorem}[Validity of canonical formula for Medvedev frames]\label{thm: validity of canonical formula for Medvedev frames}
    For each canonical formula $\alpha(P,D)$, we have:
    \begin{equation*}
        \alpha(P,D)\notin \ML \iff \mathrm{Cov}_{\ML}(P,D).
    \end{equation*}
\end{theorem}
\begin{proof}
    If $\mathrm{Cov}_{\ML}(P,D)$ holds, then by Theorem \ref{thm: validity of canonical formulas via posets}, 
    for some $n$, $\up(M_{n})\nvDash \alpha(P,D)$. Conversely, suppose that $\alpha(P,D)\notin \ML$. Since $\ML$ is the logic of the finite Medvedev frames, there is some $M_{n}$ such that $\up(M_{n})\nvDash \alpha(P,D)$. Thus by the same theorem, there is an upset $S\subseteq M_{n}$ and a cofinal partial Esakia morphism onto $f\colon S\to P$.

    By surjectivity, let $p\in S$ be such that $f(p)=r$ where $r$ is the root of $P$. Note that by the back condition, then $f{\restriction}_{{\uparrow}p}\colon {\uparrow}p\to P$ is still a cofinal partial Esakia morphism (all conditions are hereditary). Moreover, if for some $x\in {\uparrow}p$, $x\notin \domm(f_{{\uparrow}p})$, then $x\notin \domm(f)$, and so $\min f[{\uparrow}x]\notin \mathfrak{D}_{D}$. By noting that $f[{\uparrow}x]=f_{{\uparrow}p}[{\uparrow}x]$ we conclude that the latter morphism also satisfies the CDC with respect to $\mathfrak{D}_{D}$. The proof of $\mathrm{Cov}_{\ML}(P,D)$ is then concluded by noting that by Proposition \ref{prop: facts about medvedev posets}.\eqref{eq: hereditariness of Medvedev posets}, ${\uparrow}p\cong M_{m}$ for some $m$. 
\end{proof}

\begin{remark}[Medvedev posets and simplices]

It is well-known that Medvedev posets $M_{n}$ can be interpreted as simplices. For each point $p\in M_{n}$, its depth (i.e., the maximum cardinality of chains above $p$) is the \emph{dimension} of the simplex. We therefore think of the singletons from $M_{n}$ as the \emph{vertices}, the elements of the second layer as the \emph{edges}, the elements of the third layer as the \emph{triangles}, and so on, with each element being understood as a \emph{face}.

This connection and perspective will be employed in the proof below to provide some intuition on the tiling constructions, as many of these constructions can be best understood from the point of view of geometric transformations.
\end{remark}

Similarly, we can obtain similar statements about Skvortsov's logic. For that purpose, we first recall the following fact:

\begin{theorem}[Infinite frame completeness for $\sko$]\label{thm: infinite frame completeness}
    The logic $\sko$ is complete with respect to $M_{\omega}$.
\end{theorem}
\begin{proof}
    It was shown in \cite[Theorem 3.2]{Shehtman1986} that if $X$ is a  $T_{1}$ space which is not countably compact, then the logic $\sko$ is complete with respect to the Kripke frame $(\mathsf{Cl}(X),\supseteq)$. By taking the discrete space $\mathbb{N}$ we obtain the result.
\end{proof}

\begin{definition}[Skvortsov cover]
    Given $P$ a finite rooted poset, and $D\subseteq \up(P)^{2}$ we denote by $\mathrm{Cov}_{\ML}(P,D)$ the statement: there is a cofinal partial Esakia morphism onto $M_{\omega}\rightarrow P$ satisfying the CDC for $\mathfrak{D}_{D}$.
\end{definition}

\begin{theorem}[Validity of canonical formula for the infinite Medvedev frame]\label{thm: validity of canonical formula for skortsov frames}
    For each canonical formula $\alpha(P,D)$, we have:
    \begin{equation*}
        \alpha(P,D)\notin \sko \iff \mathrm{Cov}_{\sko}(P,D).
    \end{equation*}
\end{theorem}
\begin{proof}
    This follows with simpler arguments, as in Theorem \ref{thm: validity of canonical formula for Medvedev frames}, using Theorem \ref{thm: infinite frame completeness}. 
\end{proof}

\subsection{Tiling reductions}\label{sec: tiling reduction}

Tiling problems originate with Wang~\cite{Wang1961, Wang1963}.\footnote{In these papers, Wang was concerned with the un/decidability of the $\forall\exists\forall$-fragment of first-order logic, but also, as is less well-known but fun and here appropriate, with artificial intelligence. We cite: ``*insert most funny/appropriate citation*''} There are a number of different tiling problems, corresponding to and complete for various complexity and computability classes, see e.g. \cite{EmdeBoas1983,Harel1986}. We shall here be concerned with two: the \textit{(ordinary) tiling problem} and the \textit{periodic tiling problem}. Both decision problems have in common the definition of a \emph{(Wang) tile}, which is a square with coloured sides fixed in orientation, and both take as input a finite set of such tiles $\mathcal{W}$. The ordinary tiling problem then asks whether the plane $\mathbb{Z}\times \mathbb{Z}$ admits a tiling by copies of tiles from $\mathcal{W}$, with adjacent tiles matching in colour on common sides. The periodic tiling problem, in turn, does not ask whether the plane admits \textit{any} such tiling but whether it admits a \textit{periodic} one.

Since the periodic tiling problem is r.e.-complete (i.e. $\Sigma^0_1$-complete), we use it to obtain undecidability of $\ML$, as refutability in $\ML$ is recursively enumerable; and similarly, since the latter is co-r.e.-complete (i.e. $\Pi^0_1$-complete), we use it for undecidability of $\sko$. The pertinent definitions are as follows.

\begin{definition}[Wang tile]
    A \textit{(Wang) tile} is a tuple $t=(t_W, t_E, t_N, t_S)\in \mathbb{N}^4$. 
\end{definition}
\begin{definition}[Ordinary tilings]
    Given a finite set of tiles $\mathcal{W}\subseteq \mathbb{N}^4$, a \emph{$\mathcal{W}$-tiling} of the plane is a map $\tau: \mathbb{Z}\times \mathbb{Z}\to \mathcal{W}$ satisfying the horizontal and vertical compatibility conditions for all integers $k,l$:
    \begin{itemize}
        \item $\tau_E(k,l)=\tau_W(k+1,l)$ \hfill (horizontal matching)
        \item $\tau_N(k,l)=\tau_S(k,l+1)$ \hfill (vertical matching)
    \end{itemize}
\end{definition}

\begin{definition}[Periodic tilings]
    Given a finite set of tiles $\mathcal{W}\subseteq \mathbb{N}^4$, a \emph{periodic $\mathcal{W}$-tiling} consists of a pair of natural numbers $m,n$ and a map $\tau:\mathbb{Z}_m\times \mathbb{Z}_n\to \mathcal{W}$ satisfying the horizontal and vertical compatibility conditions where addition is computed modulo $m$ and $n$ in the first and second coordinate, respectively; in particular, 
    \begin{align*}
        \tau_E(m-1,l)=\tau_W(0,l)\qquad \text{and}\qquad \tau_E(k,n-1)=\tau_W(k,0).
    \end{align*}
\end{definition}

It is easy to see that the ordinary tiling problem is co-recursively enumerable. A celebrated result of Berger~\cite{Berger1966-BERTUO-6} shows that it is even co-r.e.-complete.

\begin{theorem}[The tiling problem]
    The ordinary tiling problem is co-r.e.-complete, so undecidable. That is, there is no Turing machine that, when given the specifications of a finite tile set $\mathcal{W}$, decides whether the plane admits a $\mathcal{W}$-tiling.
\end{theorem}

Conversely, is it easy to see that the periodic tiling problem is recursively enumerable. \textcite{GurevichKoryakov1972} proved its r.e.-completeness.

\begin{theorem}[The periodic tiling problem]\label{thm: tiling problem is undecidable}
    The periodic tiling problem is r.e.-complete, so undecidable. That is, there is no Turing machine that, when given the specifications of a finite tile set $\mathcal{W}$, decides whether the plane admits a periodic $\mathcal{W}$-tiling.
\end{theorem}

\section{Simplex of a tiling and tiling posets}\label{sec: tiling posets}

\begin{remark}[Restriction to surjective parity assignments]\label{rem: restriction}

For technical reasons we work with the following analogue of the periodic tiling problem: given four finite tile sets $\mathcal{W}_{i,j}$ for $i,j\in\{0,1\}$, decide whether there are is a periodic tiling $\tau:\mathbb{Z}_m\times \mathbb{Z}_n\to \bigcup \mathcal{W}_{i,j}$ with the additional constraint that, for each $i,j\in\{0,1\}$, the restriction of $\tau$ to 
\[
    \{(k,l)\in \mathbb{Z}_m\times \mathbb{Z}_n\mid k = i\mod 2,\; l=j\mod 2\}
\]
is a surjective function onto $\mathcal{W}_{i,j}$.

It is readily verified that this problem is r.e.-complete too. Similarly for the corresponding analogue for the ordinary tiling problem.

For notational convenience we henceforth denote $\mathcal{W}=\bigcup \mathcal{W}_{i,j}$ and typically talk about $\mathcal{W}$-tilings, rather than $\bigcup\mathcal{W}_{i,j}$-tilings.
\end{remark}

Our key construction will associate to a given set of Wang tiles $\W=\bigcup \mathcal{W}_{i,j}$ a poset $P_{\W}$, and a closed domain $\mathfrak{D}_{\W}$ such that:
\begin{equation}\label{eq:keyequation}
    \tag{$\dagger$} \W \text{ tiles the plane periodically }\iff \mathrm{Cov}_{\ML}(P_{\W},\mathfrak{D}_{\W}).
\end{equation}

From this the desired undecidability result follows:

\begin{proposition}\label{prop: entailment for undecidability}
    Assume that \eqref{eq:keyequation} holds. Then $\ML$ is undecidable.
\end{proposition}
\begin{proof}
    By Theorem \ref{thm: validity of canonical formula for Medvedev frames}, given any finite set of tiles $\W$ as input, whether $\mathrm{Cov}_{\ML}(P_{\W},\mathfrak{D}_{\W})$ holds is reducible to validity of a formula in $\ML$; thus by \eqref{eq:keyequation}, this provides a reduction of the tiling problem to validity in $\ML$, which is not possible by Theorem \ref{thm: tiling problem is undecidable}.
\end{proof}

To make the construction more understandable, we will proceed in several steps, approximating the goal, and proving some preliminary properties. Then in the next section we will supply all missing pieces to obtain the above equation. 

\begin{remark}[Adding fusions simplifies matters]
As mentioned in the introduciton,~\cite{Knudstorp2025} proves undecidability for a conservative extension of $\sko_\Box$. One may therefore wonder whether those results could be extended to cover the present undecidabiliy results. As the reader may guess from the fact that the present paper exists, it is our understanding that they do not.
\end{remark}

% Add somewhere in this section some explanation why the techniques from \cite{knudstorpundecidabilityrelevant} do not immediately work; motivating the difficulties, and the need for the technicalities can, I think help the reader contextualize the place where real difficulties happen.

\subsection{Encoding tiling}

To start, let us consider a tiling, as in Figure \ref{fig:periodicwangtiling}. 

\begin{figure}[h]
    \centering
\begin{tikzpicture}
\draw[help lines, color=gray!30, dashed] (-1.9,-1.9) grid (4.9,4.9);

\fill[blue!30] (0,0) -- (0.5,0.5) -- (1,0) -- (0,0) -- cycle;
\fill[green!30] (1,0) -- (1.5,0.5) -- (2,0) -- (1,0) -- cycle;
\fill[blue!30] (2,0) -- (2.5,0.5) -- (3,0) -- (2,0) -- cycle;
\fill[yellow!30] (3,0) -- (3.5,0.5) -- (4,0) -- (3,0) -- cycle;

\node at (4.5,1.5) {$\dots$};

\fill[red!30] (0,0) -- (0.5,0.5) -- (0,1) -- (0,0) -- cycle;
\fill[orange!30] (1,0) -- (1.5,0.5) -- (1,1) -- (1,0) -- cycle;
\fill[yellow!30] (2,0) -- (2.5,0.5) -- (2,1) -- (2,0) -- cycle;
\fill[red!30] (3,0) -- (3.5,0.5) -- (3,1) -- (3,0) -- cycle;

\fill[green!30] (0,1) -- (0.5,0.5) -- (1,1) -- (0,1) -- cycle;
\fill[blue!30] (1,1) -- (1.5,0.5) -- (2,1) -- (1,1) -- cycle;
\fill[orange!30] (2,1) -- (2.5,0.5) -- (3,1) -- (2,1) -- cycle;
\fill[blue!30] (3,1) -- (3.5,0.5) -- (4,1) -- (3,1) -- cycle;

\fill[orange!30] (1,0) -- (0.5,0.5) -- (1,1) -- (1,0) -- cycle;
\fill[yellow!30] (2,0) -- (1.5,0.5) -- (2,1) -- (2,0) -- cycle;
\fill[red!30] (3,0) -- (2.5,0.5) -- (3,1) -- (3,0) -- cycle;
\fill[green!30] (4,0) -- (3.5,0.5) -- (4,1) -- (4,0) -- cycle;

%%%

\fill[green!30] (0,1) -- (0.5,1.5) -- (1,1) -- (0,1) -- cycle;
\fill[blue!30] (1,1) -- (1.5,1.5) -- (2,1) -- (1,1) -- cycle;
\fill[orange!30] (2,1) -- (2.5,1.5) -- (3,1) -- (2,1) -- cycle;

\fill[blue!30] (3,1) -- (3.5,1.5) -- (4,1) -- (3,1) -- cycle;

%%%

\fill[blue!30] (0,2) -- (0.5,2.5) -- (1,2) -- (0,2) -- cycle;
\fill[green!30] (1,2) -- (1.5,2.5) -- (2,2) -- (1,2) -- cycle;
\fill[blue!30] (2,2) -- (2.5,2.5) -- (3,2) -- (2,2) -- cycle;

\fill[yellow!30] (3,2) -- (3.5,2.5) -- (4,2) -- (3,2) -- cycle;

\fill[green!30] (0,3) -- (0.5,3.5) -- (1,3) -- (0,3) -- cycle;
\fill[blue!30] (1,3) -- (1.5,3.5) -- (2,3) -- (1,3) -- cycle;
\fill[orange!30] (2,3) -- (2.5,3.5) -- (3,3) -- (2,3) -- cycle;

\fill[blue!30] (3,3) -- (3.5,3.5) -- (4,3) -- (3,3) -- cycle;

%%%SIDE PIECES%%%

\fill[orange!30] (0,1) -- (0.5,1.5) -- (0,2) -- (0,1) -- cycle;
\fill[orange!30] (1,1) -- (1.5,1.5) -- (1,2) -- (1,1) -- cycle;
\fill[yellow!30] (2,1) -- (2.5,1.5) -- (2,2) -- (2,1) -- cycle;
\fill[red!30] (3,1) -- (3.5,1.5) -- (3,2) -- (3,1) -- cycle;

\fill[red!30] (0,2) -- (0.5,2.5) -- (0,3) -- (0,2) -- cycle;
\fill[orange!30] (1,2) -- (1.5,2.5) -- (1,3) -- (1,2) -- cycle;
\fill[yellow!30] (2,2) -- (2.5,2.5) -- (2,3) -- (2,2) -- cycle;
\fill[red!30] (3,2) -- (3.5,2.5) -- (3,3) -- (3,2) -- cycle;

\fill[orange!30] (0,3) -- (0.5,3.5) -- (0,4) -- (0,3) -- cycle;
\fill[orange!30] (1,3) -- (1.5,3.5) -- (1,4) -- (1,3) -- cycle;
\fill[yellow!30] (2,3) -- (2.5,3.5) -- (2,4) -- (2,3) -- cycle;
\fill[red!30] (3,3) -- (3.5,3.5) -- (3,4) -- (3,3) -- cycle;

%%%TOP PARTS

\fill[blue!30] (0,2) -- (0.5,1.5) -- (1,2) -- (0,2) -- cycle;
\fill[green!30] (1,2) -- (1.5,1.5) -- (2,2) -- (0,2) -- cycle;
\fill[blue!30] (2,2) -- (2.5,1.5) -- (3,2) -- (2,2) -- cycle;
\fill[yellow!30] (3,2) -- (3.5,1.5) -- (4,2) -- (3,2) -- cycle;

\fill[green!30] (0,3) -- (0.5,2.5) -- (1,3) -- (0,3) -- cycle;
\fill[blue!30] (1,3) -- (1.5,2.5) -- (2,3) -- (0,3) -- cycle;
\fill[orange!30] (2,3) -- (2.5,2.5) -- (3,3) -- (2,3) -- cycle;
\fill[blue!30] (3,3) -- (3.5,2.5) -- (4,3) -- (3,3) -- cycle;

\fill[blue!30] (0,4) -- (0.5,3.5) -- (1,4) -- (0,4) -- cycle;
\fill[green!30] (1,4) -- (1.5,3.5) -- (2,4) -- (1,4) -- cycle;
\fill[blue!30] (2,4) -- (2.5,3.5) -- (3,4) -- (2,4) -- cycle;
\fill[yellow!30] (3,4) -- (3.5,3.5) -- (4,4) -- (3,4) -- cycle;

%%%SIDE PIECES

\fill[orange!30] (1,1) -- (0.5,1.5) -- (1,2) -- (1,1) -- cycle;
\fill[yellow!30] (2,1) -- (1.5,1.5) -- (2,2) -- (2,1) -- cycle;
\fill[red!30] (3,1) -- (2.5,1.5) -- (3,2) -- (3,1) -- cycle;
\fill[green!30] (4,1) -- (3.5,1.5) -- (4,2) -- (4,1) -- cycle;

\fill[orange!30] (1,2) -- (0.5,2.5) -- (1,3) -- (1,2) -- cycle;
\fill[yellow!30] (2,2) -- (1.5,2.5) -- (2,3) -- (2,2) -- cycle;
\fill[red!30] (3,2) -- (2.5,2.5) -- (3,3) -- (3,2) -- cycle;
\fill[green!30] (4,2) -- (3.5,2.5) -- (4,3) -- (4,2) -- cycle;

\fill[orange!30] (1,3) -- (0.5,3.5) -- (1,4) -- (1,3) -- cycle;
\fill[yellow!30] (2,3) -- (1.5,3.5) -- (2,4) -- (2,3) -- cycle;
\fill[red!30] (3,3) -- (2.5,3.5) -- (3,4) -- (3,3) -- cycle;
\fill[green!30] (4,3) -- (3.5,3.5) -- (4,4) -- (4,3) -- cycle;

\draw[->,ultra thick] (-2,0)--(5,0) node[right]{$x$};
\draw[->,ultra thick] (0,-2)--(0,5) node[above]{$y$};

\draw[thick] (1,0)--(1,5);

\draw[thick] (0,1)--(5,1);

\draw[thick] (2,0)--(2,5);
\draw[thick] (0,2)--(5,2);

\draw[thick] (3,0)--(3,5);
\draw[thick] (0,3)--(5,3);

\draw[thick] (4,0)--(4,5);
\draw[thick] (0,4)--(5,4);

\node at (1,-0.3) {$x_{1}$};

\node at (2,-0.3) {$x_{2}$};

\node at (3,-0.3) {$x_{3}$};

\node at (4,-0.3) {$x_{4}$};

\node at (-0.3,1) {$y_{1}$};

\node at (-0.3,2) {$y_{2}$};

\node at (-0.3,3) {$y_{3}$};

\node at (-0.3,4) {$y_{4}$};

\end{tikzpicture}    \caption{Part of a periodic Wang tiling on the positive plane}
    \label{fig:periodicwangtiling}
\end{figure}
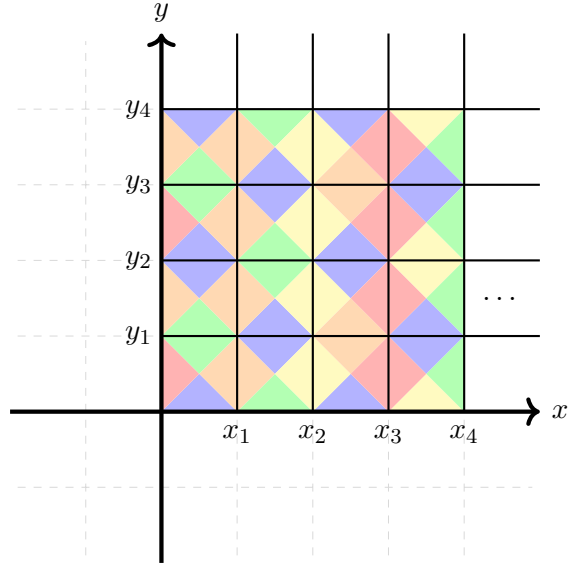

To be able to prove \eqref{eq:keyequation}, we need to record several pieces of data:
\begin{enumerate}
    \item the compatibility structure of the tiles, as well as
    \item whether or not tiling is possible, and
    \item make it so that the existence of a tiling can be turned into a p-morphic image out of a Medvedev frame, and vice-versa.
\end{enumerate}

To obtain some intuition about this procedure, let us begin by assuming that $\mathcal{W}$ tiles the plane periodically. As mentioned in Section \ref{sec: tiling reduction}, a periodic tiling can be formally represented by a function $\tau\colon \mathbb{Z}_{m}\times \mathbb{Z}_{k}\to \W$, also called the ``periodic rectangle" of the tiling. To encode it, we can take the points on each axis, from which the pairs are recoverable. But to talk about ``adjacency", we additionally record the \emph{parity} of the points.

\begin{definition}[Simplex of a tiling $\tau$]\label{def: simplex of tiling}
    For each set of Wang tiles, and each periodic tiling $\tau\colon \mathbb{Z}_{m}\times \mathbb{Z}_{k}\to \W$, we define the horizontal sets:
    \begin{equation*}
        \Ho_{0}\coloneqq \{n\in \mathbb{Z}_{m} : n \text{ is even}\} \text{ and } \Ho_{1}\coloneqq \{n\in \mathbb{Z}_{m} : n \text{ is odd} \};
    \end{equation*}
    similarly we define:
    \begin{equation*}
        \V_{0}\coloneqq \{n\in \mathbb{Z}_{k} : n \text{ is even}\} \text{ and } \V_{1}\coloneqq \{n\in \mathbb{Z}_{k} : n \text{ is odd} \}.
    \end{equation*}
    We let $\Omega_{\tau}\coloneqq \Ho_{0}\sqcup \Ho_{1}\sqcup \V_{0}\sqcup \V_{1}$, and let $\til(\Omega)\coloneqq \mathcal{P}(\Omega_{\tau})\setminus \{\varnothing\}$ be the Medvedev frame, called the \emph{simplex of a tiling}  $\tau$.
\end{definition}

\begin{notation}[Coordinate-parity assignment]
    For a point $n\in \Omega_{\tau}$ we let $\mathrm{cor}(n)\in \{\Ho_{0},\Ho_{1},\V_{0},\V_{1}\}$ be the unique set to which it belongs. We denote by $\mathrm{cor}[F]$ the usual extension of this to subsets.
\end{notation}

The idea of the simplex of a tiling is that one can look inside of $\til(\Omega)$ to obtain information about the tiling. Namely:
\begin{enumerate}
    \item The doubletons $\{n,n'\}$ where $n\in \mathbb{Z}_{m},n'\in \mathbb{Z}_{k}$ record \emph{tiling information} -- they tell us precisely to which tile the pair is sent to.
    \item On the other hand the doubletons $\{n,n'\}$ where exactly one of the two elements is in $\Ho_{0}$ and the other in $\Ho_{1}$ (and similarly for vertical sets) capture the \emph{successor relations} between points: either this is a set of the form $\{n,n+1\}$, in one of the two possible ways it can come to be, or it is not.
\end{enumerate}

\begin{remark}[Not all faces have obvious meaning]
    Of course there are other faces in the simplex whose meaning is less clear: a subset $C\subseteq \Ho_{0}$ does not obviously have any role in the tiling. Nevertheless, as we will see, such faces play an important role in encoding the possibility of having a periodic tiling.
\end{remark}

Once we know what information needs to be kept, we will restrict essentially to a (quotient of) a sub-poset of $\til(\Omega)$, obtaining the desired reduction. With this in mind, we introduce the following construction:

\begin{definition}[Tiling poset of $\W$]
    Given a set $\W$ of tiles, the poset $P_{\W}$ contains the following points:
    \begin{enumerate}
        \item (Horizontal parity points) Two maximal points $\Ho_{0}$ and $\Ho_{1}$;
        \item (Vertical parity points) Two maximal points $\V_{0}$ and $\V_{1}$.
        \item (Horizontal successor points) Three points $\R_{t}$, $\R_{01}$ and $\R_{10}$ that have exactly $\Ho_{0}$ and $\Ho_{1}$ as its immediate successors.
        \item (Vertical successor points) Three points $\So_{t}$, $\So_{01}$ and $\So_{10}$ that have exactly $\V_{0}$ and $\V_{1}$ as its immediate successors.
        \item (Tile points) For each $i, j\in \{1,0\}$ and each $t\in \W_{i,j}$ a tile point $\T_{i,j,t}$ that has $\Ho_{i}$ and $\V_{j}$ as its immediate successors.
        \item (Simplex points) For each maximal point $m$, a simplex point $s_{m}$ which has only $m$ as its immediate successor;
        \item (Cut point) For each maximal point $m$, a cut point $c_{m}$ containing only $s_{m}$ as its successor.
        \item (Root) A root element $r$ which is below all other points.
    \end{enumerate}
\end{definition}

For $\mathcal{W}=\{t\}$ a singleton tile, the construction looks as in Figure \ref{fig:tilingposetsingletile} (all points below the root; most edges removed for clarity).

\begin{figure}[h]
    \centering

    \begin{tikzpicture}
        \node (h0) at (0,0) {$\Ho_{0}$};
        \node (s0) at (-1,-1) {$s_{\Ho_0}$};
        \node (c0) at (-1,-2) {$c_{\Ho_0}$};

        \node (h1) at (4,0) {$\Ho_{1}$};
        \node (s1) at (3,-1) {$s_{\Ho_1}$};
        \node (c1) at (3,-2) {$c_{\Ho_1}$};

        \node (r00) at (0,-1) {$\R_{00}$};
        \node (r01) at (1,-1) {$\R_{01}$};
        \node (r10) at (2,-1) {$\R_{10}$};

        \draw (h0) -- (r00) -- (h1);
        \draw (h0) -- (r01) -- (h1);
        \draw (h0) -- (r10) -- (h1);

        \node (v0) at (8,0) {$\V_{0}$};
        \node (sv0) at (9,-1) {$s_{\V_0}$};
        \node (cv0) at (9,-2) {$c_{\V_0}$};

        \node (v1) at (12,0) {$\V_{1}$};
        \node (sv1) at (13,-1) {$s_{\V_1}$};
        \node (cv1) at (13,-2) {$c_{\V_1}$};

        \node (s00) at (12,-1) {$\So_{00}$};
        \node (s01) at (11,-1) {$\So_{01}$};
        \node (s10) at (10,-1) {$\So_{10}$};

        \draw (v0) -- (s00) -- (v1);
        \draw (v0) -- (s01) -- (v1);
        \draw (v0) -- (s10) -- (v1);

        \draw (h0) -- (s0) -- (c0);
        \draw (h1) -- (s1) -- (c1);
        \draw (v0) -- (sv0) -- (cv0);
        \draw (v1) -- (sv1) -- (cv1);

        \node (t00) at (4.5,-1) {$\T_{00}$};
        \node (t01) at (5.5,-1) {$\T_{01}$};
        \node (t10) at (6.5,-1) {$\T_{10}$};
        \node (t11) at (7.5,-1) {$\T_{11}$};

\draw (h0) -- (t00) -- (v0);
\draw (h0) -- (t01) -- (v1);
\draw (h1) -- (t10) -- (v0);
\draw (h1) -- (t11) -- (v1);

\node (r) at (6,-3) {$r$};

\draw (c0) -- (r) -- (c1) -- (r) -- (cv0) -- (r) -- (cv1);

    \end{tikzpicture}    \caption{Poset $P_{\W}$ for $\W=\{t\}$}
    \label{fig:tilingposetsingletile}
\end{figure}
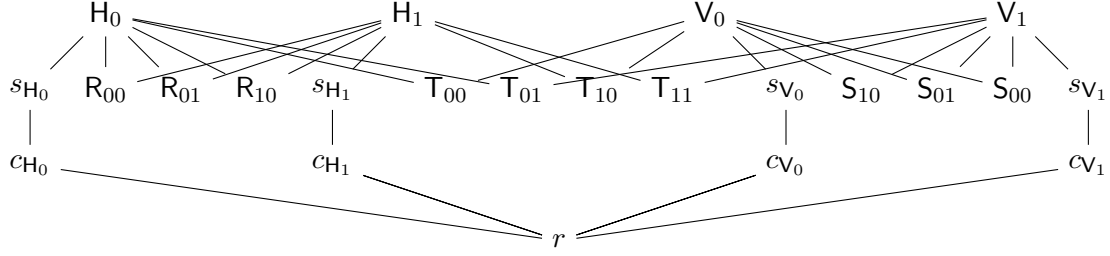

The terminology chosen reflects the roles that each kind of point has. Such roles are explained in part by how one then constructs the map from $\til(\Omega)$ back to $P_{\W}$. The following formalizes the idea of ``recording data" just made:

\begin{definition}[Folding morphism on $\tau$]\label{def: folding morphism}
    Given a tiling $\tau$ of the set of tiles $\W$, we define the folding map $f_{\tau}\colon \til(\Omega)\to P_{\W}$ by induction, as follows:
    \begin{enumerate}
        \item (Singleton) For each singleton $\{n\}$, $f_{\tau}(\{n\})=\mathrm{cor}(n)$;
        \item For each doubleton $D=\{n,n'\}$:
        \begin{enumerate}
            \item If $f_{\tau}(\{n\})=\Ho_{0}$ and $f_{\tau}(\{n'\})=\Ho_{1}$, then:
            \begin{enumerate}
                \item $f_{\tau}(D)=\R_{01}$ if $n'=n+1$ $\mod {m}$;
                \item $f_{\tau}(D)=\R_{10}$ if $n=n'+1 \mod {m}$;
            \item $f_{\tau}(D)=\R_{t}$ otherwise.\end{enumerate}
            
        \item If $f_{\tau}(\{n\})=\V_{0}$ and $f_{\tau}(\{n'\})=\V_{1}$, then:
        \begin{enumerate}
            \item $f_{\tau}(D)=\So_{01}$ if $n'=n+1 \mod {k}$;
            \item $f_{\tau}(D)=\So_{10}$ if $n=n'+1 \mod {k}$;
        \item $f_{\tau}(D)=\So_{t}$ otherwise.
        \end{enumerate}
        \item If $f_{\tau}(\{n\})=\Ho_{i}$ and $f_{\tau}(\{n'\})=\V_{j}$ then let $t\coloneqq \tau(n,n')$ and map $f_{\tau}(D)=\T_{i,j,t}$.
\end{enumerate}
\item If $F\subseteq \Omega$ is any face such that $|F|>1$ (including a doubleton), such that $|\mathrm{cor}[F]|=T$ for $T\in \{\Ho_{0},\Ho_{1},\V_{0},\V_{1}\}$, then:
\begin{enumerate}
    \item If $F$ contains all points sent to $T$, then $f_{\tau}(F)=c_{T}$;
    \item Otherwise $f_{\tau}(F)=d_{T}$.
\end{enumerate}  
        \item $f_{\tau}(\Omega)=r$.
        \item $f_{\tau}$ is not defined anywhere else on the domain of $\til(\Omega)$.
    \end{enumerate}
\end{definition}

In other words, the points $\R$ record the information about horizontal successors, with the extra point $\R_{t}$ serving as a ``trash point"; similarly the points $\So$ record information about vertical successors; and the tile points record the structure of the tiling.

\begin{remark}[Simplex and cut points]
    
    The role of the simplex and cut points is more involved, and is best understood by thinking of the map $f_{\tau}$ as providing a \emph{coloring} to the simplex, with the top labels $\Ho_{i}$ and $\V_{j}$ serving as the colors.

    Under this reading, the points which get mapped to the \emph{simplex} point are exactly the monochromatic simplices which are neither maximal nor degenerate. Those which get sent to the \emph{cut} point are the maximal monochromatic simplices. 
    
    Having such points buys us expressive power in the tiling: since we can talk about monochromatic faces, if we ensure certain faces \emph{omit} any non-trivial monochromatic simplex, this can force us to identify certain points. On the other hand,    
    since we can talk about the largest monochromatic simplex, we can force the existence of points within a given prescribed successor relation. 
\end{remark}

We record the following basic fact about this construction:
\begin{proposition}[Domain definition]\label{prop: domain definition}
    For each set $\W$ of tiles, $\tau$ a tiling and $F\in \til(\Omega)$ we have that $F\in \domm(f_{\tau})$ if and only if one of the following holds:
    \begin{enumerate}
        \item $|F|\leq 2$;
        \item $|F|>2$ and $\mathrm{cor}[F]$ is a singleton.
        \item $F=\Omega$.
    \end{enumerate}
\end{proposition}

With the given definition of the folding morphism on a tiling we can easily prove:

\begin{proposition}[Folding map yields a cofinal partial Esakia morphism]\label{prop: folding map partial esakia}
    For each set $\W$ of tiles, and $\tau\colon \mathbb{Z}_{m}\times\mathbb{Z}_{k}\to \W$ a tiling where $m,k\geq 6$\footnote{As remarked in the next section, this assumption is purely technical and can be removed without affecting the core argument of the proof.}, the folding map $f_{\tau}$ is a cofinal partial Esakia morphism onto $P_{\W}$.
\end{proposition}
\begin{proof}
    The fact that the map is onto follows from an easy case analysis, noting that the simplex $\til(\Omega)$ must contain points in all the given positions, since $m,k\geq 6$, and since, by Remark \ref{rem: restriction}, only the tile points which arise in the tiling $\tau$ are put in the poset $\til(\Omega)$. 

Moreover, the map is order-preserving and satisfies the back-condition. The domain condition follows at once from Proposition \ref{prop: domain definition}. Finally, the map is cofinal, since all of the singletons are in the domain.    
\end{proof}

\begin{remark}[Remarks on partiality]
Proposition \ref{prop: folding map partial esakia} and its elementary proof is one of the key advantages of working with Zakharyaschev's canonical formulas: instead of having to define the function $f_{\tau}$ on a \emph{whole structure}, it is possible to define it only on the parts which carry the most relevant geometric information.
\end{remark}

\subsection{Enforcing successors and compatibility}

So far, despite the structure $\Omega_{\tau}$ depending on the period of $\tau$, the fact that $\tau$ is indeed a tiling is not properly encoded, as the next example shows:

\begin{example}[Missing compatibility]\label{ex: missing compatibility}
Consider a tripleton $\{n,n',n''\}$. By focusing on its face poset we obtain, for example, the picture of Figure \ref{fig:compatibilitytriple}:

\begin{figure}[h]
    \centering

    \begin{tikzpicture}
        \node at (0,0) {$\bullet$};
        \node at (0,0.3) {$n$};
        \node at (-1.5,-2.5) {$\bullet$};
        \node at (-1.8,-2.8) {$n'$};
        \node at (1.5,-2.5) {$\bullet$};
        \node at (1.8,-2.8) {$n''$};

        \draw (0,0) -- (-1.5,-2.5) -- (1.5,-2.5) -- (0,0);

        \node at (0,-3) {$R_{10}$};

        \node at (-1.5,-1.3) {$T_{t,1,0}$};

        \node at (1.5,-1.3) {$T_{t',0,0}$};
\end{tikzpicture}    \caption{Compatibility triple}
    \label{fig:compatibilitytriple}
\end{figure}
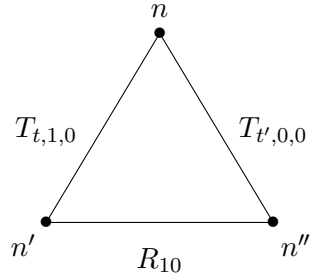
Given the way the map $f_{\tau}$ was designed, we know that $n''=n'+1$, and so the horizontal coordinates align, and moreover, $\tau(n',n)=t$ whilst $\tau(n'',n)=t'$. However, there is an obvious additional requirement: $t$ should be \emph{compatible} with $t'$, i.e.
\begin{equation*}
    t_E=t'_W,
\end{equation*}
and similarly for vertical compatibility.
\end{example}

The issue this points to is the fact this whole process should be reversible -- so that from a cofinal partial Esakia morphism, one can construct a tiling. Indeed, not only is the compatibility such an issue, but we additionally need to ensure that from such a map we have enough doubletons mapping appropriately to the $\R$ and $\So$ points.

It is in both of these situations that closed domains, as well as the special simplex and cut points, can help us; we give the following definition, where we remark in passing that all of the subsets considered are upsets.

\begin{definition}[Tiling domain of $\W$]\label{def: tiling domain}
    Given a set $\W$ of tiles, we define the following closed domain $\mathfrak{D}_{\W}$ of $P_{\W}$:
    \begin{enumerate}
    \item\label{eq:maximalcuts} (Maximal cuts) For $(A,B)$ one of the six pairs:
    \begin{equation*}
        (\Ho_{0},\Ho_{1}),(\V_{0},\V_{1}),(\Ho_{0},\V_{0}),(\Ho_{0},\V_{1}),(\Ho_{1},\V_{0}),(\Ho_{1},\V_{1}),
    \end{equation*}
    we add the pair $(\{A\},\{B\})$.
        \item (Successor cuts) Given a pair $(S,T)$, from amongst:
        \begin{equation*}
            (\Ho_{0},\Ho_{1}),(\Ho_{1},\Ho_{0}),(\V_{0},\V_{1}),(\V_{1},\V_{0}),
        \end{equation*}
        we use the following definitions:

\begin{table}[h]
    \centering
\begin{tabular}
{@{}llll@{}}
\toprule
Pair & Target \(T\) & Non-matching \(N\) & Matching \(Q\)\\
\midrule
$(\Ho_{0},\Ho_{1})$ & \(\mathsf H_1\)
  & \(\{R_{00},R_{10}\}\) & \(R_{01}\)\\
$(\Ho_{1},\Ho_{0})$ & \(\mathsf H_0\)
  & \(\{R_{00},R_{01}\}\) & \(R_{10}\)\\
$(\V_{0},\V_{1})$ & \(\mathsf V_1\)
  & \(\{S_{00},S_{10}\}\) & \(S_{01}\)\\
$(\V_{1},\V_{0})$ & \(\mathsf V_0\)
  & \(\{S_{00},S_{01}\}\) & \(S_{10}\)\\
\bottomrule
\end{tabular}
    \caption{Successor cut data}
    \label{tab:successorcutdata}
\end{table}

Using this table, we define the following sets:        \begin{equation*}
            A_{(S,T)}=P_{\W}\setminus (N_{(S,T)}\cup \{Q_{(S,T)},r\}) \text{ and } B_{(S,T)}=P_{\W}\setminus (\{Q_{(S,T)},c_{T},r\}).
        \end{equation*}
        and we put the pair $(A_{(S,T)},B_{(S,T)})\in D_{\W}$.
  \item (Horizontal compatibility cuts) For each $i,j\in\{0,1\}$ and each $t\in \W_{i,j}$, let $\mathrm{HComp}(t)=\{t'\in \W_{1-i,j} : t_E=t'_W\}$, and shorten $\R_{0}\coloneqq \R_{01}$ and $\R_{1}\coloneqq \R_{10}$. Then define:
        \begin{equation*}
            \mathrm{G}_{i,j,t}\coloneqq \{\T_{1-i,j,t'} \mid t'\in \mathrm{HComp}(t)\}.
        \end{equation*}
        and define the following sets:
        \begin{align*}
            C_{i,j,t}&\coloneqq P_{\W}\setminus (\{\T_{i,j,t},s_{\Ho_{i}},c_{\Ho_{i}},r\}\cup \mathrm{G}_{i,j,t})\\
            D_{i,j,t} &\coloneqq P_{\W}\setminus (\{\R_{i},s_{\Ho_{i}},c_{\Ho_{i}},r\}\cup \mathrm{G}_{i,j,t}).
        \end{align*}
        and we put the pair $(C_{i,j,t},D_{i,j,t})\in D_{\W}$.
\item (Vertical compatibility cuts) For each $i,j\in\{0,1\}$ and each $t\in \W_{i,j}$, let $\mathrm{VComp}(t)=\{t'\in \W_{i,1-j} \mid t_N=t'_S\}$, and shorten $\So_{0}\coloneqq \So_{01}$ and $\So_{1}\coloneqq \R_{10}$. Then define:
        \begin{equation*}
            \mathrm{G}_{i,j,t}\coloneqq \{\T_{i,1-j,t'} : t'\in \mathrm{VComp}(t)\}.
        \end{equation*}
        and define the following sets:
        \begin{align*}
            C_{i,j,t}&\coloneqq P_{\W}\setminus (\{\T_{i,j,t},s_{\V_{i}},c_{\V_{i},r}\}\cup \mathrm{G}_{i,j,t})\\
            D_{i,j,t} &\coloneqq P_{\W}\setminus (\{\So_{i},s_{\V_{i}},c_{\V_{i}},r\}\cup \mathrm{G}_{i,j,t}),
        \end{align*}
    and we put the pair $(C_{i,j,t},D_{i,j,t})\in D_{\W}$.
\end{enumerate}

    We then set $\mathfrak{D}_{\W}$ as in Definition \ref{def: closed domain condition}, i.e., adding to $\mathfrak{D}_{\W}$ all antichains which are contained in the union of one of the pairs, and have non-empty intersection with both sets.
\end{definition}

Before explaining the intuition behind these definitions, we provide some simple calculations with these cuts; all such calculations are completely elementary, requiring, essentially, enough patience to untangle the notation.

\begin{lemma}[Cut calculations]\label{lem: cut calculations}
    Given any set of tiles $\W$ the following holds:
    \begin{enumerate}
        \item  For each pair of \emph{successor cuts} $(S,T)$, we have that:
        \begin{enumerate}
            \item $A_{(S,T)}\setminus B_{(S,T)}=\{c_{T}\}$;
            \item $B_{(S,T)}\setminus A_{(S,T)}=N_{(S,T)}$;
            \item $P_{\W}\setminus (A_{(S,T)}\cup B_{(S,T)})=\{r,Q_{(S,T)}\}$.
        \end{enumerate}
        \item\label{eq:horizontal cut calculations} For each tile $t\in \W$, and each pair $(C_{i,j,t},D_{i,j,t})$ of \emph{horizontal compatibility cuts} we have:
        \begin{enumerate}
            \item $C_{i,j,t}\setminus D_{i,j,t}=\{\R_{i}\}$;
            \item $D_{i,j,t}\setminus C_{i,j,t}=\{\T_{i,j,t}\}$;
            \item $P_{\W}\setminus (C_{i,j,t}\cup D_{i,j,t})=\{r,s_{\Ho_{i}},c_{\Ho_{i}}\}\cup \mathrm{G}_{i,j,t}$.
        \end{enumerate}
        \item For each tile $t\in \W$, and each pair $(C_{i,j,t},D_{i,j,t})$ of \emph{vertical compatibility cuts} we have:
        \begin{enumerate}
            \item $C_{i,j,t}\setminus D_{i,j,t}=\{\So_{i}\}$;
            \item $D_{i,j,t}\setminus C_{i,j,t}=\{\T_{i,j,t}\}$;
            \item $P_{\W}\setminus (C_{i,j,t}\cup D_{i,j,t})=\{r,s_{\V_{i}},c_{\V_{i}}\}\cup \mathrm{G}_{i,j,t}$.
        \end{enumerate}
    \end{enumerate}
\end{lemma}

\begin{example}[Example \ref{ex: missing compatibility} revisited]
    Let us again return to the triangle of Figure \ref{fig:compatibilitytriple}. We will show, simply from the fact that $f_{\tau}$ satisfies the CDC with respect to $\mathfrak{D}_{\W}$, that $t$ and $t'$ are compatible. For this purpose, consider the tripleton:
    \begin{equation*}
        F=\{n,n',n''\}.
    \end{equation*}
    Given what we assume of this face, and the way that $f_{\tau}$ was defined, we have that $F\notin \domm(f_{\tau})$. On the other hand, by assumption, we have:
    \begin{equation}\label{eq: image of the map}
        f_{\tau}[{\uparrow}F]=\{\R_{10},\T_{t,1,0},\T_{t',0,0}\}.
    \end{equation}
    By construction this is an antichain, and we have by the assumption of the CDC that it does not belong to $\mathfrak{D}_{\W}$. Now assume, towards a contradiction, that $t'\notin \mathrm{HComp}(t)$ i.e., the tile is not horizontally compatible. Then $\T_{t',0,0}\notin \mathrm{G}_{i,j,t}$, and so by Lemma \ref{lem: cut calculations}.\eqref{eq:horizontal cut calculations}:
    \begin{equation*}
        f[{\uparrow}F]\subseteq C_{i,j,t}\cup D_{i,j,t}.
    \end{equation*}
    On the other hand, by the same calculations, one sees, upon inspecting \eqref{eq: image of the map} that we have $f_{\tau}[{\uparrow}F]\cap (C_{i,j,t}\setminus D_{i,j,t})\neq \emptyset$ and $f_{\tau}[{\uparrow}F]\cap (D_{i,j,t}\setminus C_{i,j,t})\neq \emptyset$. This is a contradiction to the construction of $\mathfrak{D}_{\W}$, since all antichains contained in the union and having non-empty intersection with both upsets. By reductio, we conclude that $t'$ is horizontally compatible.
\end{example}

The idea contained in Example \ref{ex: missing compatibility} is precisely the key drive of the proof: to ensure that points are compatibly related, one constructs a \emph{cut} of the frame into two upsets, such that intersecting their complement means precisely the condition we are after. The simple example above will then recur in the proofs below.

\section{Undecidability of Medvedev's logic}\label{sec: undecidability of med}

Having in the previous section set up the basic constructions, we now move to prove the two sides of the equivalence \eqref{eq:keyequation}. The two directions are independent, and proceed by making formal a lot of the discussion from Section \ref{sec: tiling posets}.

\subsection{Tiling \texorpdfstring{$\Rightarrow$}{implies} Medvedev cover}

Assume that $\W$ is a set of tiles, and $\tau\colon \mathbb{Z}_{m}\times\mathbb{Z}_{k}\to \W$ is a periodic tiling.\footnote{Strictly speaking, we are, of course, assuming according to Remark~\ref{rem: restriction}, a particular kind of tiling.} Without loss of generality\footnote{One can simply repeat the period six times to ensure this is possible. This assumption is technical and only needed for the surjectivity.} we assume that $m,k\geq 6$ (so that for each of the elements in $\{\Ho_{0},\Ho_{1},\V_{0},\V_{1}\}$, the simplex and cut points are hit). Our goal is to construct a cofinal partial Esakia morphism onto $P_{\W}$ satisfying the CDC with respect to $\mathfrak{D}_{\W}$.

As described in Definition \ref{def: simplex of tiling} we construct the simplex $\til(\Omega)$, and as in Definition 
\ref{def: folding morphism} take the folding map $f_{\tau}$. Then we show: 

\begin{proposition}[Folding map satisfies CDC]\label{prop: folding map satisfies CDC}
    The map $f_{\tau}$ is a cofinal partial Esakia morphism onto $P_{\W}$ satisfying the CDC with respect to $\mathfrak{D}_{\W}$.
\end{proposition}
\begin{proof}
    The fact that $f_{\tau}$ is a cofinal partial Esakia morphism and onto was shown in Proposition \ref{prop: folding map partial esakia}. We thus focus on proving the CDC.

    So assume that $F\in \til(\Omega)$ is a face such that $F\notin \domm(f_{\tau})$. Thus, by Proposition \ref{prop: domain definition} we know that $F$ is not the root, $|F|>2$ and there exist at least two points $n,n'$ with different coordinate-parity.      The proof proceeds by cases, depending on the kind of cut considered.

    \begin{enumerate}
        \item (Horizontal cuts) Suppose towards a contradiction that $f_{\tau}[{\uparrow}F]$ is an antichain contained in $C_{i,j,t}\cup D_{i,j,t}$, and intersecting each of the two sets. The latter condition means that there are $D_{i}\subseteq F$ such that $f_{\tau}(D_{i})=\R_{i}$, and $D_{t}\subseteq F$ such that $f_{\tau}(D_{t})=\T_{i,j,t}$. Say that $D_{i}=\{n_{0},n_{1}\}$. Assume without loss of generality that  $f_{\tau}(\{n_{0}\})=\Ho_{0}$ and $f_{\tau}(\{n_{1}\})=\Ho_{1}$. The construction of $f_{\tau}$ then tells us  that $n_{1}=n_{0}+1$. 
        
        On the other hand, by definition of $D_{t}=\{n_{2},n_{3}\}$ where $f_{\tau}(\{n_{2}\})=\Ho_{0}$ ($i=0$ by the assumption on $\R_{i}=\R_{01}$ made above), and $f_{\tau}(\{n_{3}\})=\V_{j}$. 

        However, note that by Lemma \ref{lem: cut calculations}, $s_{\Ho_{i}}\notin C_{i,j,t}\cup D_{i,j,t}$. Since by assumption we must have $f_{\tau}(\{n_{0}\})=f_{\tau}(\{n_{2}\})$ we conclude:
        \begin{equation*}
            n_{0}=n_{2}
        \end{equation*}
        in other words, the \emph{simplex point being excluded ensures the square becomes a triangle along the $i$-coordinate}.
        
        We thus consider the tripleton:
        \begin{equation*}
            F_{0}=\{n_{0},n_{1},n_{3}\}\subseteq F.
        \end{equation*}
        The situation obtained now is analogous to that from Example \ref{ex: missing compatibility}. Namely: $f_{\tau}(\{n_{0},n_{1}\})=\R_{01}$, and $f_{\tau}(\{n_{0},n_{3}\})=\T_{i,j,t}$.  So we can ask what is $f_{\tau}(\{n_{1},n_{3}\})$ is. By construction on the tiling, we have that $\tau(n_{0},n_{3})=t$ and $\tau(n_{1},n_{3})=t'$, and since $n_{1}=n_{0}+1$, we have that
        \begin{equation*}
            \mathrm{w}(t)=\mathrm{e}(t').
        \end{equation*}
        Thus $f_{\tau}(\{n_{1},n_{3}\})=\T_{1,j,t'}$, and by construction, $\T_{1,j,t'}\in \mathrm{G}_{i,j,t}$. But this is a contradiction, since by Lemma \ref{lem: cut calculations}, $\mathrm{G}_{i,j,t}\cap (C_{i,j,t}\cup D_{i,j,t})=\varnothing$. This is a contradiction to our assumption on $f_{\tau}[{\uparrow}F]$ being contained in the union.
        
        An entirely similar case follows if we assume that $i=1$.
        \item (Vertical cuts) The proof is entirely analogous to the case for horizontal cuts.
        \item (Successor cuts) Suppose that $(S,T)$ is a pair of successor cuts, say, without loss of generality that it is $(\Ho_{0},\Ho_{1})$. Assume that $f_{\tau}[{\uparrow}F]$ is an antichain contained in $A_{(\Ho_{0},\Ho_{1})}\cup B_{(\Ho_{0},\Ho_{1})}$, and intersecting each of them as needed. Note that then $F_{1}\subseteq F$ is a face such that every odd number from $\mathbb{Z}_{m}$ is contained there (by definition of when $f_{\tau}(F_{1})=c_{\Ho_{1}}$). On the other hand there is a doubleton $D\subseteq F$ such that $f_{\tau}(D)=\R_{00}$ or $f_{\tau}(D)=\R_{10}$, and so there is a point $n\in F$ such that $f_{\tau}(\{n\})=\Ho_{0}$. Let $n'$ be arbitrary such that $n'=n+1\mod \mathbb{Z}_{m}$; then $n'\in F$ by our assumption on $F_{1}$. Thus we have that:
        \begin{equation*}
            D_{0}=\{n,n'\}\subseteq F,
        \end{equation*}
        and by definition $f_{\tau}(D_{0})=\R_{01}$, whilst by assumption, $\R_{01}\notin A_{(\Ho_{0},\Ho_{1})}\cup B_{(\Ho_{0},\Ho_{1})}$. This is a contradiction to $f_{\tau}[{\uparrow}F]$ being contained in the union.
    \end{enumerate}
    
\end{proof}

This establishes one direction:

\begin{corollary}[Periodic tiling implies Medvedev covers]\label{cor: Tiling implies Medvedev covers}
    Whenever $\W$ is a set of tiles which tiles periodically, then $\mathrm{Cov}_{\ML}(P_{\W},\mathfrak{D}_{\W})$ holds. 
\end{corollary}

\subsection{Medvedev cover \texorpdfstring{$\Rightarrow$}{implies} tiling}\label{sec: medvedev cover implies tiling}

We now focus on the harder direction: showing that from a cofinal partial Esakia morphism satisfying CDC, we can extract a periodic tiling of the plane.

Fix $f\colon M_{n}\to P_{\W}$ such a partial map. First let:
\begin{equation*}
    A\coloneqq \max (f^{-1}[\{\Ho_{0},\Ho_{1}\}]) \text{ and } B\coloneqq \max(f^{-1}[\{\V_{0},\V_{1}\}]).
\end{equation*}
Note that $A,B$ partition $\max(M_{n})$, since $f$ is cofinal.

The proof will proceed in a series of lemmas:

\begin{lemma}[Well-defined induced tiling function]\label{lem: well-defined induced tiling}
    For each $a\in A$ and $b\in B$, we have that
    \begin{equation*}
        f(\{a,b\})\in \{\T_{i,j,t} : i,j\in \{0,1\},t\in \W\},
    \end{equation*}
    i.e., $f$ is defined at the doubleton, and it maps to one of the tile points.
\end{lemma}
\begin{proof}
    Note that if $f$ is defined on $\{a,b\}$, then to satisfy order-preservation, it must map to the tile points. If $\{a,b\}\notin \domm(f)$, since $f$ satisfies the CDC with respect to $\mathfrak{D}_{\W}$, we obtain that $f[{\uparrow}\{a,b\}]\notin \mathfrak{D}_{\W}$ -- a contradiction to the latter containing all antichains in the maximal cuts (see Definition \ref{def: tiling domain}.\eqref{eq:maximalcuts}). By reductio, we conclude that $\{a,b\}\in \domm(f)$.
\end{proof}
Using Lemma \ref{lem: well-defined induced tiling}, we obtain that given $a\in A$ and $b\in B$, $f(\{a,b\})=\T_{i,j,t}$; we denote by $t_{a,b}$ the tile ocurring in the last position of the subscript. We then set:
\begin{equation*}
    \tau(a,b)\coloneqq t_{a,b}.
\end{equation*}

By an analogous argument, the following holds as well.

\begin{lemma}[Doubletons are in the domain]\label{lem: doubletons are in the domain}
Let
\begin{equation*}
    A\coloneqq \max (f^{-1}[\{\Ho_{0}\}]) \text{ and } B\coloneqq \max(f^{-1}[\{\Ho_{1}\}]).
\end{equation*}
    Then for each $a\in A$ and $b\in B$, we have that
    \begin{equation*}
        f(\{a,b\})\in \{\R_{01}, \R_{10}, \R_{t}\},
    \end{equation*}
    i.e., $f$ is defined at the doubleton, and it maps to one of the horizontal successor points. 

    Similarly, let 
    \begin{equation*}
    A\coloneqq \max (f^{-1}[\{\V_{0}\}]) \text{ and } B\coloneqq \max(f^{-1}[\{\V_{1}\}]).
\end{equation*}
Then for each $a\in A$ and $b\in B$, we have that
    \begin{equation*}
        f(\{a,b\})\in \{\So_{01}, \So_{10}, \So_{t}\},
    \end{equation*}
    i.e., $f$ is defined at the doubleton, and it maps to one of the vertical successor points. 
\end{lemma}

We will now show that this provides a well-defined tiling:

\begin{lemma}[Successor lemma]\label{lem: successor lemma}
    For each $a\in A$ and $b\in B$, if $f(\{a\})=\Ho_{i}$, there is some $a'\in A$ such that $f(\{a'\})=\Ho_{1-i}$, and such that:
    \begin{equation*}
        f(\{a,a'\})=\R_{i1-i}.
    \end{equation*}   
    Similarly for each $a\in A$ and $b\in B$ such that $f(\{b\})=\V_{j}$, there exists some $b'\in B$ such that $f(\{b'\})=\V_{1-j}$, and
    \begin{equation*}
        f(\{b,b'\})=\So_{j1-j}.
    \end{equation*}
\end{lemma}
\begin{proof}
    The two statements follow from very similar arguments, using vertical instead of horizontal cuts and successors. We show the case for horizontal successors in detail.
    
    Without loss of generality assume that $f(\{a\})=\Ho_{0}$. We once again use the cut points: since $f$ is onto, pick $F\in M_{n}$ such that $f(F)=c_{\Ho_{1}}$. Then look at:
    \begin{equation*}
        F'\coloneqq F\cup \{a\}.
    \end{equation*}
    Note that since $f(F)=c_{\Ho_{1}}$ and $f(\{a\})=\Ho_{0}$, and the map is order-preserving, if $F'$ was in the domain, we would have to have $f(F')=r$; but since $f$ is a p-morphism, then $F'$ would have to contain points mapping to $\V_{0},\V_{1}$ -- a contradiction. So we obtain that $F'\notin \domm(f)$. Because $f$ satisfies the CDC with respect to the successor cut $(\Ho_{0},\Ho_{1})$, this means that:
    \begin{equation*}
        \min f[{\uparrow}F']\notin \mathfrak{D}_{\W}.
    \end{equation*}
    Take $e\in F$, then $f[\{e\}]=\Ho_{1}$; thus by Lemma \ref{lem: doubletons are in the domain}, we know that $f[\{e,a\}]\in \{\R_{t},\R_{10},\R_{01}\}$. If $f[\{e,a\}]=\R_{01}$, we are done, so suppose not. Then we have:
    \begin{enumerate}
        \item $f[{\uparrow}F']\cap A_{(\Ho_{0},\Ho_{1})}\setminus B_{(\Ho_{0},\Ho_{1})}\neq \emptyset$, since $f(F)=c_{T}$, and $F'\supseteq F$;
        \item $f[{\uparrow}F']\cap B_{(\Ho_{0},\Ho_{1})}\setminus A_{(\Ho_{0},\Ho_{1})}\neq \emptyset$, since one of $\R_{t}$ and $\R_{10}$ is by hypothesis there.
    \end{enumerate}
    
    Since we have that $\min f[{\uparrow}F']\notin \mathfrak{D}_{\W}$ this can only happen if $\min f[{\uparrow}F']\nsubseteq A_{(\Ho_{0},\Ho_{1})}\cup B_{(\Ho_{0},\Ho_{1})}$. Thus either $r\in \min f[{\uparrow}F']$ -- which we have ruled out -- or $\R_{01}\in  f[{\uparrow}F']$. By taking $F'\supseteq E$, where $E$ is any face such that $f(E)=\R_{01}$, we see that there must be a point $a'\in E$, and by order-preservation, and the fact that $E\supseteq \{a,a'\}$, we get that $f(\{a,a'\})=\R_{01}$.
\end{proof}

We call an element $a'$ (respectively $b'$) satisfying the conditions of Lemma \ref{lem: successor lemma} a \emph{successor} of $a$. Importantly, we have that, \emph{independently of which successor is chosen in this way}, the tiling is always appropriately compatible:

\begin{lemma}[Compatibility lemma]\label{lem: key compatibility lemma}
    Let $a,a'\in A$ be such that $f(\{a,a'\})=\R_{i1-i}$. Then for any $b$, we have:
\begin{equation*}
        \mathrm{e}(t_{a,b})=\mathrm{w}(t_{a',b}).
    \end{equation*}
    Similarly, whenever $b,b'\in B$ are such that $f(\{b,b'\})=\So_{j1-j}$, then for any $a\in A$, we have:
    \begin{equation*}
        \mathrm{n}(t_{a,b})=\mathrm{s}(t_{a,b'}).
    \end{equation*}
\end{lemma}
\begin{proof}
    Once again we show the case for horizontal compatibility, the vertical one following by exactly the same arguments. 

    We need to show that $t_{a,b}$ and $t_{a,b'}$ are compatible. For this purpose, we once again consider the tripleton:
    \begin{equation*}
        E=\{a,a',b\}.
    \end{equation*}
    Note that $f(\{a\})=\Ho_{0}$, $f(\{a'\})=\Ho_{1}$ and $f(\{b\})=\V_{j}$, so $E\notin \domm(f)$ by the same arguments as given in Lemma \ref{lem: successor lemma} for why $F'\notin \domm(f)$. Since $f$ satisfies the CDC with respect to $\mathfrak{D}_{\W}$, we have that:
    \begin{equation*}
        f[{\uparrow}E]\notin \mathfrak{D}_{\W} \text{ and } f[{\uparrow}E]=\{\R_{01},\T_{0,j,t_{a,b}},\T_{1,j,t_{a',b}}\}.
    \end{equation*}
    Once again, by Lemma \ref{lem: cut calculations}, we have that $f[{\uparrow}E]\cap C_{0,j,t_{a,b}}\setminus D_{0,j,t_{a,b}}\neq \varnothing$, and we have likewise that $f[{\uparrow}E]\cap D_{0,j,t_{a,b}}\setminus C_{0,j,t_{a,b}}\neq \varnothing$. Thus $f[{\uparrow}E]\nsubseteq C_{0,j,t_{a,b}}\cup D_{0,j,t_{a,b}}$, which means that:
    \begin{equation*}
        f[{\uparrow}E]\cap \mathrm{G}_{0,j,t_{a,b}}\neq \varnothing.
    \end{equation*}
    This means that $\T_{1,j,t_{a',b}}\in \mathrm{G}_{0,j,t_{a,b}}$; by construction of the latter set, we then have that $\mathrm{e}(t_{a,b})=\mathrm{w}(t_{a',b})$, which was to show.
\end{proof}

Using Lemma \ref{lem: successor lemma}, we can conclude our desired result:

\begin{proposition}[Induced tiling function is a periodic tiling]\label{prop: induced tiling function is a periodic tiling}
    The function $\tau$ defines a periodic tiling of the plane.
\end{proposition}
\begin{proof}
    Pick $a_{0},b_{0}$ two arbitrary elements. Using Lemma \ref{lem: successor lemma}, we pick a chain of successors:
    \begin{equation*}
        a_{0},a_{1},\dots,a_{n},\dots
    \end{equation*}
    Since $M_{n}$ is finite, and the points $a_{i}$ as above are maximal points of $M_{n}$, we have that eventually there must be a period -- and thus we may assume that in fact $a_{0},\dots,a_{m}$ is the first such chain. Note moreover that on the $b_{0}$-row, this provides a tiling by Lemma \ref{lem: key compatibility lemma}. 

    Now pick $b_{1}$ some successor of $b_{0}$. Then note that again by Lemma \ref{lem: key compatibility lemma}, the tiles $\tau(a_{i},b_{0})$ and $\tau(a_{i},{b_{1}})$ will be vertically compatible, whilst the tiles $\tau(a_{i},b_{1})$ and $\tau(a_{i+1},b_{1})$ will be horizontally compatible. Proceeding in this way, we again obtain a period $b_{1},\dots,b_{k}$, where all tiles are pairwise compatible as needed.
\end{proof}

This entails at once the desired result:

\begin{corollary}[Medvedev covers implies periodic tiling]\label{cor: Medvedev covers implies tiling}
    For each set of tiles $\W$, whenever $\mathrm{Cov}_{\ML}(P_{\W},\mathfrak{D}_{\W})$, we have that $\W$ tiles the plane periodically.
\end{corollary}

This concludes the proof of the equivalence \eqref{eq:keyequation}. Thus we conclude:

\begin{theorem}\label{thm: medvedev's logic is undecidable}
    Medvedev's logic $\ML$ is undecidable.
\end{theorem}
\begin{proof}
    By Proposition \ref{prop: entailment for undecidability}, we have that proving \eqref{eq:keyequation} shows that $\ML$ is undecidable. The left to right direction is Corollary \ref{cor: Tiling implies Medvedev covers}, whilst the right to left direction is Corollary \ref{cor: Medvedev covers implies tiling}.
\end{proof}

\section{Undecidability of Skvortsov's logic}\label{sec: undecidability skortsov}

We now turn our attention to the logic $\sko$. As we will see, the methods of the previous two sections adapt very naturally to the study of $\sko$, and reveal a striking fact about intuitionistic validities: the distinction between $\sko$ and $\ML$ -- which we will prove exists -- lies exactly in the existence of periodic tilings.

To understand this, we begin by considering the following basic relationship:

\begin{equation}\label{eq:secondkeyequation}
    \tag{$\dagger\dagger$} \W \text{ tiles the plane }\iff \mathrm{Cov}_{\sko}(P_{\W},\mathfrak{D}_{\W}).
\end{equation}

Just like with Proposition \ref{prop: entailment for undecidability}, we obtain:

\begin{proposition}\label{prop: equation and skvortsov undecidability}
    Assume that \eqref{eq:secondkeyequation} holds. Then $\sko$ is undecidable.
\end{proposition}

We briefly record the key changes to the construction:

\begin{enumerate}
    \item (Simples of a tiling): in  Definition \ref{def: simplex of tiling}, by simply changing $\tau$ to be a tiling $\tau\colon\mathbb{N}\times\mathbb{N}\to \W$, and proceeding in the same way, we obtain a frame $\til(\Omega)$, where $\Omega\cong \mathbb{N}$, and thus, $\til(\Omega)\cong M_{\omega}$. Similarly the construction of the tiling domain of $\W$ from Definition \ref{def: tiling domain} remains unchanged.
    \item The construction of the tiling poset $\W$ is done in exactly the same way (note that $\W$ is finite regardless).
    \item The folding morphism of Definition \ref{def: folding morphism} can still be carried out -- note that the only modification necessary is on ``faces" which may now be infinite, but where, nevertheless, one can consider the largest face with a given coordinate parity (by taking unions), sending that one to $c_{T}$ for any $T\in \{\Ho_{0},\Ho_{1},\V_{0},\V_{1}\}$.
\end{enumerate}

Moreover, we have:

\begin{proposition}[Folding map satisfies CDC]\label{prop: folding map satisfies CDC infinite}
    If $\tau$ is an arbitrary tiling map $f_{\tau}$ is a cofinal partial Esakia morphism onto $P_{\W}$ satisfying the CDC with respect to $\mathfrak{D}_{\W}$.
\end{proposition}
\begin{proof}
    We can repeat the proofs of Proposition \ref{prop: folding map partial esakia} and Proposition \ref{prop: folding map satisfies CDC} verbatim, except one swaps periodic tiling for usual tiling.
\end{proof}

\begin{corollary}[Periodic tiling implies Skvortsov covers]\label{cor: Tiling implies skortsov covers}
    Whenever $\W$ is a set of tiles which tiles the plane, then $\mathrm{Cov}_{\sko}(P_{\W},\mathfrak{D}_{\W})$ holds. 
\end{corollary}

In the opposite direction, fix $f\colon M_{\omega}\to P_{\W}$ a cofinal partial Esakia morphism satisfying the CDC with respect to $\mathfrak{D}_{\W}$. We proceed as in Section \ref{sec: medvedev cover implies tiling}, setting $A=f^{-1}[\{\Ho_{0},\Ho_{1}\}]$ and $B=f^{-1}[\{\V_{0},\V_{1}\}]$, and noting that $A\cong B\cong \mathbb{N}$ as sets.

Now observe that Lemma \ref{lem: successor lemma} and Lemma \ref{lem: key compatibility lemma} do not depend in any way on the finiteness of the frame. Thus, we have a function $\tau$ induced by $f$, where:
\begin{equation*}
    \tau(a,b)=t_{a,b}.
\end{equation*}

We then check the following:

\begin{proposition}[Induced tiling function is a tiling]\label{prop: induced tiling function is a tiling}
    The function $\tau$ defines a tiling of the plane.
\end{proposition}
\begin{proof}
    We proceed as in Proposition \ref{prop: induced tiling function is a periodic tiling}. By picking $a_{0}$ and $b_{0}$ arbitrary, we produce an enumeration:
    \begin{equation*}
        A_{0}\coloneqq (a_{0},a_{1},a_{2},\dots,a_{n},\dots,)
    \end{equation*}
    and likewise, enumerate,
    \begin{equation*}
        B_{0}\coloneqq (b_{0},b_{1},b_{2},\dots,b_{n},\dots)
    \end{equation*}
    where we construct the enumeration row-by-row using Lemma \ref{lem: successor lemma}. Lemma \ref{lem: key compatibility lemma} ensures that the pairs are always compatible, meaning that restricting to the domain $A_{0}\times B_{0}$ obtains a tiling of the plane, as desired.    
\end{proof}

\begin{corollary}[Skortsov covers implies tiling]\label{cor: skortsov covers implies tiling}
    For each set of tiles $\W$, whenever $\mathrm{Cov}_{\sko}(P_{\W},\mathfrak{D}_{\W})$, we have that $\W$ tiles the plane.
\end{corollary}

From this we obtain:

\begin{theorem}\label{thm: skortov is undecidable}
    Skvortsov's logic $\sko$ is undecidable.
\end{theorem}
\begin{proof}
    By Proposition \ref{prop: equation and skvortsov undecidability}, we have that $\sko$ being undecidable is equivalent to \eqref{eq:secondkeyequation}. The left to right direction is supplied by Corollary \ref{cor: skortsov covers implies tiling}, whilst the right to left is given by Corollary \ref{cor: skortsov covers implies tiling}.
\end{proof}

By inspecting the proofs just provided, we also obtain the following interesting corollary:

\begin{corollary}[$\sko\subsetneq \ML$]\label{cor: separation of skvortsov and medvedev}
    The logic $\sko$ is a proper sublogic of $\ML$.
\end{corollary}
\begin{proof}
    Let $\W$ be a finite set of Wang tiles which is known to \emph{only tile the plane} aperiodically\footnote{For an example of such a tile, see \cite{Jeandel2021}, which also contains extensive discussion on such tile sets.}. By Corollary \ref{cor: Medvedev covers implies tiling}, we have that $\mathrm{Cov}_{\ML}(P_{\W},\mathfrak{D}_{\W})$ fails, and thus, by Thoerem \ref{thm: validity of canonical formula for Medvedev frames}, $\alpha(P_{\W},\mathfrak{D}_{\W})\in \ML$. On the other hand, by Corollary \ref{cor: skortsov covers implies tiling}, we have that $\mathrm{Cov}_{\sko}(P_{\W},\mathfrak{D}_{\W})$, which 
    by Theorem \ref{thm: validity of canonical formula for skortsov frames}, means that $\alpha(P_{\W},\mathfrak{D}_{\W})\notin \sko$.
\end{proof}

\section{Discussion and open problems}\label{sec: discussions and open problems}

The undecidability Theorems \ref{thm: medvedev's logic is undecidable} and \ref{thm: skortov is undecidable}, and the separation made in Corollary \ref{cor: separation of skvortsov and medvedev}, rather than closing the topic of decidability of these logics, present in our view, an attractive open area of research for a wide variety of logical questions. As discussed in Section \ref{sec: methodology and context}, we believe that a community wide effort, eventually with the use of tools such as large language models, can help in advancing our understanding of the fundamental questions at play.

\subsection{Superintuitionistic logics}

The negative answers to Problems \ref{prob:recursive axiomatizability}, \ref{prob: skortsov logic decidability},  \ref{prob: coincidence problem} and \ref{prob: decidability of modal medvedev}, substantial as they are, still leave many questions open about $\ML$ and its status amongst superintuitionistic logics. The first and foremost is Problem \ref{prob: yankov axiomatizability} of Yankov axiomatizability. Indeed, consider the logic
\begin{equation*}
    Y_{\mathsf{ML}}\coloneqq \mathsf{IPC}\oplus \{\mathcal{J}(P) : P\nVdash \mathsf{ML}\}.
\end{equation*}

Shehtman's question leads to wondering whether any of the following holds:

\begin{problem}[Axiomatizability of the Yankov logic of $\ML$]
    Is $Y_{\ML}$ equal to $\ML$? Is it is equal to $\sko$? Is it equal to $Y_{\sko}$? Is this logic decidable?
\end{problem}

Despite our efforts, both human and AI-assisted, we have not found a concrete way to separate any of these logics\footnote{See the Section \ref{sec: methodology and context} for a link to the research files on these investigations.}. We note that such a problem has a very natural combinatorial-geometric interpretation: if we see an arbitrary poset $P$ as an alexandroff space, the decidability question asks whether, given an arbitrary poset $P$, it is decidable whether a simplex has an open surjection onto $P$. Such an open surjection admits an obvious reading as a ``coloring", though the precise details need to be further explored.

A further refinement of this problem can be obtained by considering the notion of \emph{degree of FMP}:

\begin{definition}[Degree of FMP]
    A superintuitionistic logic $L$ is said to have \emph{degree of the finite model property $\kappa$} for $\kappa$ a cardinal if and only if the following set has cardinality $\kappa$\footnote{Note that it was proven in \cite[Thm.~3.8,Rem.~3.9]{cardinalitiesdegreeoffmp} that for the degree of FMP (and many other measures in lattices of varieties), the only possible infinite cardinals are $\aleph_{0}$ and $2^{\aleph_{0}}$.}:
    \begin{equation*}
        \mathrm{deg}_{\mathsf{FMP}}(L)\coloneqq |\{L_{0}\in \mathrm{Ext}(\mathsf{IPC}) : L_{0} \text{ has the same finite frames as } L\}|.
    \end{equation*}
\end{definition}

As shown in \cite[Thm.~3.11]{Bezhanishvili2025blok}, a logic has degree of FMP 1 if and only if it has the FMP and is Yankov axiomatizable. So one can more broadly ask:

\begin{problem}[Degree of FMP of $\ML$]
    Characterize the degree of FMP of $\ML$.
\end{problem}

In \cite{xiaogeneralizedmedvedev} a generalized class of Medvedev logics was introduced; this is obtained as follows: given a finite rooted poset $F$ with a top element, write $F^{n}$ for the poset obtained by the product $F\times\dots \times F$, $n$ many times, and let $F^{n-}$ be the same poset with the top point removed. Then the \emph{generalized Medvedev logic over $F$} $\mathsf{TLP}_{\mathbf{F}}$ is the logic
\begin{equation*}
    \mathsf{TLP}_{\mathbf{F}}\coloneqq \mathrm{Log}(\{F^{n-} : n\in \omega\}).
\end{equation*}

One can therefore ask:

\begin{problem}[Generalized Medvedev decidability]
    Does there exist a finite poset $F$ with a top element such that $\mathsf{TLP}_{\mathbf{F}}$ is decidable?
\end{problem}

Such a research can also assist in a deeper understanding of these and related logics (like the ``Strong-Union logic" introduced in \cite{Chen2026}), in the study of a related open question concerning Medvedev's logic:

\begin{problem}[Friedman uniqueness problem]
    Is $\ML$ the unique Friedman logic (both structurally complete and with the disjunction property)?
\end{problem}

On the side of more conceptual and vague questions, as mentioned in Section \ref{sec: undecidability skortsov}, the relationship between $\ML$ and $\sko$ appears to be mediated by the possibility to express periodicity phenomena, and periodicity in general. Indeed, the analysis by Ghilardi and Santocanale of periodicity in endomorphisms \cite{Ghilardi2020},revealed in particular \cite[Sec.~8]{Ghilardi2020} the existence of non-periodic endomorphisms, witnessing a failure of a generalization of Ruitenburg's theorem. Very recently, non-periodicity phenomena appeared to play a key role in the counterexample to the general Heyting-to-topos problem \cite{failureofhigherorder}. Understanding the connection, if any, between these phenomena, appears to us an important theoretical step forward.

\subsection{Schematic fragments and team semantics}\label{subsec:schematic}

Inquisitive semantics, developed by Ciardelli, Groenendijk, Roelofsen~\cite{Ciardelli2011-CIAIL,inqsembook,ciardellibook}, provides a formal framework for reasoning about entailment relations between assertions and questions alike. This is implemented by evaluating formulas with respect to {sets} $s$ of classical propositional valuations $v:\mathsf{Prop}\to \{0,1\}$, rather than single valuations $v$. The clauses for the propositional language $\{\bot,\land, \to, \vvvee\}$ run as follows:

\begin{align*}
        &s\vDash p && \textbf{iff} && \text{for all } v\in s, v(p)=1\\
        &s\vDash \bot && \textbf{iff} && s=\varnothing\\
        &s\vDash \varphi\land\psi && \textbf{iff} && s\vDash \varphi \text{ and } s\vDash \psi\\
        &s\vDash \varphi\vvvee\psi && \textbf{iff} && s\vDash \varphi \text{ or } s\vDash \psi\\
        &s\vDash \varphi\to\psi && \textbf{iff} && \text{for all } t\subseteq s, \text{ if }t\vDash \varphi\text{ then }t\vDash \psi.
\end{align*}
Negation $\neg$ is, as usual, defined as $\neg\varphi\mathrel{:=}\varphi\to\bot$, and $\top$ as $\bot\to\bot$. While the clause for $\vvvee$ mirrors the clause for disjunction $\vee$ in the relational semantics for intermediate logics, please take note that the symbol `$\vee$' is, within inquisitive semantics, reserved for the defined connective $\varphi\vee \psi\mathrel{:=}\neg (\neg \varphi \land \neg \psi)$ (this way, the $\vvvee$-free fragment, $\{\bot, \top, \land, \vee, \to\}$, can be shown to coincide with classical logic). 

We call a set of valuations $s$ an \textit{information state} or a \textit{team}, and say that it \textit{supports} a formula $\varphi$ if $s\vDash \varphi$. Validity and entailment, or consequence, is defined as preservation of support across all states. The resulting consequence relation is, like superintuitionistic logics, closed under modus ponens, but unlike superintuitionistic logics, is it not closed under uniform substitution. Indeed, it can be readily verified that 
\begin{align*}
    \neg\neg p\vDash p, \qquad \text{yet} \qquad \neg\neg(p\vvvee \neg p)\nvDash (p\vvvee \neg p).
\end{align*}
It is thus natural to ponder the \emph{schematic} fragment of inquisitive logic: the set of consequences, or formulas (inquisitive logic enjoys the deduction property), that are not only valid, but whose substitution instances are all valid, i.e. the set 
\[
    \mathsf{Sch}(\mathsf{Inq})\mathrel{:=}\{\varphi\mid \text{for all substitutions $\sigma{:}\; \vDash \sigma(\varphi)$}\}.
\]
As mentioned in the introduction, this coincides precisely with Medvedev's logic~\cite{Ciardelli2011-CIAIL}. The undecidability of $\ML$ thereby yields a stark contrast with the decidability of inquisitive logic (i.e. of validity alone, contrary to schematic validity). The decidability of inquisitive validity follows by an entirely elementary proof: to check that a formula $\varphi$ of inquisitive logic is valid, it suffices to check whether all states $s\subseteq \{v:\mathsf{Prop}(\varphi)\to\{0,1\}\}$ support $\varphi$. Here, $\mathsf{Prop}(\varphi)$ denotes the set of propositional letters occurring in $\varphi$. This is a finite set, whence there are only finitely many classical valuations with domain $\mathsf{Prop}(\varphi)$, and therefore only finitely many subsets of such, and so decidability follows.

This decidability argument generalizes to a whole range of kindred logical systems, going under the collective name of propositional (or, for that matter, modal) \textit{team logics}; for a small sample, see~\cite{yangvaananen2016,yangvaananen2017}. The undecidability of Medvedev's logic therefore opens an unprecedented research programme within team semantics: what otherwise seemed like a remarkably robust family of logics, is carrying features of high complexity within. Obvious, and exciting, questions involve understanding this much, much better. As an immediate next step, one may seek to adapt the present proof to the most closely related logics, such as propositional dependence logic or convex inquisitive logic~\cite{yangvaananen2016,AnttilaKnudstorp}.

\subsection{Related logics}

Finally we turn to several logics which are close in spirit or in shape to $\ML$, and for which the questions of their decidability has been open for as long, or longer, than that of $\ML$.

As noted before, the logic $\mathsf{Cheq}$ was introduced to analyze the study of \emph{chequered subsets} of the real numbers, a concept arising in the modal logical analysis of space. It is defined as follows: let $\mathcal{F}$ be the 2-fork poset, which we label as $\mathcal{F}_{0}$. We define by recursion $\mathcal{F}_{n+1}\coloneqq \mathcal{F}_{n}\times \mathcal{F}$ (see Figure \ref{fig:framesF0andF1}).

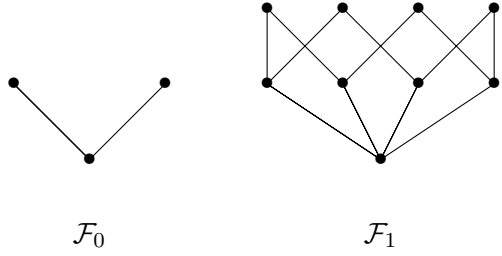
\begin{figure}[h]
    \centering
\begin{tikzpicture}
    \node at (0,0) {$\bullet$};
    \node at (-1,1) {$\bullet$};
    \node at (1,1) {$\bullet$};
    \node at (0,-1) {$\mathcal{F}_{0}$};

\draw (0,0) -- (-1,1) -- (0,0) -- (1,1);

\end{tikzpicture}
\qquad
\begin{tikzpicture}
\node at (0,0) {$\bullet$};

    \node at (-0.5,1) {$\bullet$};
    \node at (-1.5,1) {$\bullet$};

    \node at (0.5,1) {$\bullet$};
        \node at (1.5,1) {$\bullet$};

    \node at (-0.5,2) {$\bullet$};
    \node at (-1.5,2) {$\bullet$};

    \node at (0.5,2) {$\bullet$};
    \node at (1.5,2) {$\bullet$};

    \draw (0,0) -- (-0.5,1) -- (-1.5,2) -- (-1.5,1) -- (-0.5,2) -- (0.5,1) -- (1.5,2) -- (1.5,1) -- (0.5,2) -- (-0.5,1) -- (0,0) -- (-1.5,1) -- (0,0) -- (0.5,1) -- (0,0) -- (1.5,1);

    \node at (0,-1) {$\mathcal{F}_{1}$};

\end{tikzpicture}    \caption{Frames $\mathcal{F}_{0}$ and $\mathcal{F}_{1}$}
    \label{fig:framesF0andF1}
\end{figure}

We then define:
\begin{equation*}
    \mathsf{Cheq}\coloneqq \mathrm{Log}(\{\mathcal{F}_{n} : n\in \omega\}).
\end{equation*}

Given how semantically similar this logic is to $\ML$, the following problem might be solved by adapting the techniques of the present paper:

\begin{problem}[$\mathsf{Cheq}$ recursively axiomatizable]
    Is $\mathsf{Cheq}$ recursively axiomatizable? Equivalently, is $\mathsf{Cheq}$ decidable?
\end{problem}

An even older\footnote{We note that Medvedev \cite{medvedevfiniteproblems} was aware of this work, and reacting in part to the work of Rose \cite[Thm.~6.1]{roserealizability} who had shown that $\mathsf{IPC}$ was not complete for this interpretation.}
relative of $\ML$ is the \emph{realizability logic} of Kleene \cite{Kleene1945-KLEOTI}. Crucial to this is the relation $n\mathbf{r}\lambda$, which expresses that $n$ \emph{realizes} the sentence $\lambda$ (in an informal sense, that $n$ encodes this sentence). We say that a Heyting arithmetic sentence $\lambda$ is realizable if $n\mathbf{r}\lambda$ is true for some $n$.

Given a propositional formula $\phi(p_{1},\dots,p_{n})$, and arithmetical formulas $\lambda_{1},\dots,\lambda_{n}$ -- that is, formulas in the language of first-order Heyting arithmetic -- we will say that $\phi(\lambda_{1},\dots,\lambda_{n})$ is an \emph{arithmetical instance} (or arithmetical realization\footnote{The reader should note that this is wholly analogous to the concept used in the provability logic literature of a ``realizer", see \cite{Boolos1993-BOOTLO-7} for a reference}) of $\phi$. We let $\mathsf{R}(\phi)$ be the set of the arithmetical realizations of $\phi$. Then we define:
\begin{equation*}
    \mathsf{Realiz}\coloneqq \{\phi\in \mathcal{L}_{\mathsf{IPC}} : \forall \lambda\in \mathsf{R}(\phi), \lambda \text{ is realizable }\}.
\end{equation*}

A survey of propositional realizability logic \cite{Plisko2009} noted that many basic facts concerning this logic, and several of its close relatives, remain open. 

\begin{problem}[Basic meta-logic of realizability logic]\label{prob: realizability logic}
    Is the logic $\mathsf{Realiz}$ decidable? Does it have the finite model property?
\end{problem}

Such a question has been noted as being wholly open in \cite{deJongh2010}. This can likewise be phrased in topos theoretic terms: recall the effective topos\cite{HYLAND1982165}  $\eff$ which serves as a model for Kleene (number) realizability. Then Problem \ref{prob: realizability logic} is equivalent to asking for the internal logic of the topos $\eff$.

\section{Methodology and Context}\label{sec: methodology and context}

In this section we discuss the process of obtaining the present results in full detail.

\subsection{Usage of large language models}

The following section outlines the usage of large language models in this paper. 

\subsubsection{Data collection and presentation}

As outlined in the abstract and introduction, the key idea for the undecidability of $\ML$ -- namely the construction of the tiling posets $P_{\W}$, and the encoding of the reduction -- was obtained by prompting ChatGPT Sol 5.6. The present text was written by the authors to try to communicate the key ideas, and presents, in our view, a more polished and transparent version of this proof. In short, whilst the authors assume \emph{responsibility} for the correctness of the results, they do not take any credit in such ideas.

To that point, we note that with the exception of the notation from Definition \ref{def: tiling domain}, which was for convenience lifted from the preliminary document, the text contained in the present paper was \textbf{written entirely by the authors, and represents faithfully our understanding of the main mathematical results}\footnote{The reader who checks the dates and timing of the prompts will hopefully apologize all of the typos which inevitably must have been introduced by the proof writing. In a subsequent version, we hope to correct such typos. Indeed, the urgency in writing was not quick enough to avoid the main result being announced earlier.}. These were checked carefully by both authors, who retain full responsibility for the correctness and presentation of the results, in accordance with the Leiden declaration. Usage of Opus 5 was made to proofread the text for typos and to improve the prose, though all changes were implemented manually by the authors.

In light of these facts, in the interest of transparency, and to aid the community in evaluating the present paper, and eventually obtaining similar results, we have made all documents involved in this process fully available -- all prompts, all preliminary documents and verification ledgers produced by the model during rounds, and a Lean verification of the main central argument, produced by Claude Opus 5\footnote{With a view to the integrity of the academic process, we have not modified any such document. As such, the first-named author registers his embarassment, both at the fact that the prompts contain some mis-recollections of facts in the literature, and that some of the writings in the preliminary documents reflect the personal relationship between the authors.} -- they can be found at the first-named author's personal website, at \href{http://rodrigonalmeida.github.io}{this link}. In the absence of standardized procedures, we believe that documents should be treated analogously to datasets in other areas, and thus, should be subject to similar policies of disclosure, reproducibility, and be cited just like any other documents. At the same time, we also believe that the production and distribution of such datasets for public usage should be encouraged, especially in view of the potential of LLMs to accentuate inequalities in access to research tools.

\subsubsection{Prompt contents}

The idea to prompt ChatGPT for a solution of Medvedev's logic arose in discussion between the authors, in light of the recent solution of another longstanding open problem \cite{failureofhigherorder}, on which the first-named author had been actively working, which was obtained with AI assistance.

The second-named author, having worked throughout his PhD work on undecidability, spent time thinking about the decidability of Medvedev's logic, grew suspicious that a well-targeted search might settle the problem in the positive. At his behest, the first-named author constructed the following prompting session: 
\begin{enumerate}
    \item The model was asked to reconstruct Zakharyaschev's canonical formula criterion for axiomatization, and to present a small simplification guaranteed by the structural completeness of $\ML$;
    \item After the model collected such results, a second prompt was engineered to try to obtain a proof of decidability or undecidability: a wide-search cross-referencing results from the literature on graph topology, combinatorics, and Constraint Satisfaction Problems, was organized in a ``positive team" whose task was to provide a decidability proof for the desired criterion. In parallel, a ``negative team" was tasked with using similar techniques to those from \cite{knudstorpundecidabilityrelevant} to obtain a tiling encoding. The prompt organizes the main workflow between agents, and contains some prompt engineering aimed to the goal of maximizing the chances the model returned a useful criterion.
\end{enumerate}

After running for 45 minutes, the model returned with an undecidability result, which it was asked to check and transcribe into a latexed pdf file. This was surprising to the authors on several grounds. One of the most remarkable was the importance of Zakharyaschev's canonical formulas: despite several attempts by the authors to fix such a construction, it quickly became clear that using the simpler \emph{Yankov formulas} in fact made the problem a lot harder\footnote{The difficulty being that such formulas force the existence of a \emph{total} function.}. Despite several attempts, using the theory of cofinitely generated Esakia spaces, several tools developed by Shehtman, and AI assistance, no result was obtained either separating or showing the decidability of the Yankov fragment of $\ML$.

Upon obtaining the undecidability proof, the second-named author asked whether a similar undecidability result could be obtained for Skvortsov's logic. The second document produced in this thread contains a summary of these results. We found the proof of it, however, to be rather immediate from the construction for $\ML$, and so the proof provided does not use anything from this latter document.

The authors also sought to prompt Chat GPT Sol 5.6 to answer Problem \ref{prob: yankov axiomatizability}. Here the model produced vastly less decisive results, despite additional care being placed in prompt engineering\footnote{The outcomes of such questions are likewise publicly available for use by the community.}.

\subsection{Prompt engineering and heuristics}

It is our belief that systematic empirical research is necessary to determine the precise conditions involved in successful prompt engineering for the solution of problems as here. Some general heuristics that we expect can be helpful are recorded here:
\begin{enumerate}
    \item \textbf{Workflow management and prompt engineering}: Organized research worfklows, with careful division of tasks, and a well developed bookkeeping, appear to be more economic (in terms of token usage) and more effective in obtaining useful reports from the agents involved in the research. As more analysis comes out about coordinating such swarms of agents, better methods can be developed to optimize, economically and cognitively, the task of obtaining solutions and processing them efficiently.
    \item \textbf{Search-space management} The effectiveness of solutions seems to be most closely related to the search-space associated to the problem. Retrospectively it is clear that the search space for decidability of Medvedev's logic was relatively small: the gap in kinds of techniques and approaches involved in a hypothetical positive solution, and the those used for the present negative solution, was not very large, and the large literature on Medvedev's logic meant that the model had a vast library from which to draw directly from its training. In other problems where the search space is much larger, additional human insight is necessary to circumscribe the precise techniques at play -- to wit, the request to prove the undecidability of Skvortsov's logic was very quickly executed.
    \item \textbf{Attributions and reconstructions}: in our experience, long-running frontier models are exceptionally good at tracking even vague allusions to connections between topics. A concrete example arose in the second preliminary document, where connections with the theory of simplicial sets were discussed. Though the model reproduced verbatim some of the results from \cite{bevilacqua2025medvedev} and \cite{jerabekinternal}, it did not credit the first, and only in passing mentioned the second files. This is to our view an unavoidable feature of the current LLM architecture: deep chains of tought and reasoning lead to models with a very poor capacity for self-recollection.
\end{enumerate}

\subsection{Community consultation}

On par with our belief that explicit methodological questions and a suitable data policy (analogous to many other fields of research), is crucial for the safe usage of LLM tools, we also believe that the use of such tools pose far-reaching questions which will need a deep reflection in the ongoing years. 

It is our belief, for the purpose of this paper, that the materials obtained and made publicly available should be the purview of all researchers interested in the logic at hand. Moreover, we believe that whilst we have striven to provide an interpretation which we hope communicates the deep theoretical insights obtained by the model, we believe the many deep questions raised in Section \ref{sec: discussions and open problems}, as well as many others the reader may be familiar with, provide the backbone of this. As such, we intend to maintain this paper in the status of a working paper, as a deeper understanding of the situation emerges, and will not submit it to a journal for publication until such a deep understanding is achieved. This reflects in our view, drawn in part from considerations laid out by Tao \cite{taomathematicsageofai}, that one of the the fundamental tasks of mathematicians, when entrusted with an AI generated result, is to slowly and deliberately digest it until the key insights have been revealed.

Our hopes are thus manifold. Locally, we hope that the interested reader will assist us in this process of unveiling the key mathematical insights at play in all of this. More broadly, we hope that the methodological and ethical outlines presented here can serve as a general basis, out of which journals, conferences and other venues in our field can develop effective criteria for judging transparency, ethics and quality of contributions. Finally, it is our generic hope that research in mathematical logic will continue to strive for a level of rigor and clarity which ensures that all contributions advance human understanding as much as possible, regardless of the specific methods used to achieve them.

\subsubsection*{Acknowledgments} 

We are grateful to Nick Bezhanishvili for comments and valuable discussions concerning the nature and presentation of the work outlined above.

\printbibliography[
    title={Bibliography}
]

\vspace{5mm}

\textbf{Rodrigo Nicolau Almeida}\\
Institute for Logic, Language and Computation\\ University of Amsterdam\\
Amsterdam, 1098XH\\
The Netherlands\\
r.dacruzsilvapinadealmeida@uva.nl

\vspace{3mm}

\textbf{Søren Brinck Knudstorp}\\
Institute for Logic, Language and Computation\\ University of Amsterdam\\
Amsterdam, 1098XH\\
The Netherlands\\
s.b.knudstorp@uva.nl

\end{document}